\documentclass[reqno]{amsart} 
\usepackage{amssymb, cite} 
\usepackage{cmtiup, mathrsfs}
\usepackage{mathtools}

\newtheorem{theorem}{Theorem}[section]
\newtheorem{lemma}{Lemma}[section] 
\newtheorem{corollary}{Corollary}[section] 
\newtheorem{proposition}{Proposition}[section] 
 
\theoremstyle{definition} 
\newtheorem{remark}{Remark}[section]

\def\Int{\operatorname{{Int}}}

\begin{document} 
\title{Free Boolean Topological Groups}

\author{Ol'ga V. Sipacheva}

\address{Department of General Topology and Geometry, Lomonosov Moscow State University}
\email{o-sipa@yandex.ru}

\begin{abstract} Known and new results on free Boolean topological groups are collected. 
 An account of properties which these groups share with free or free Abelian topological groups 
and properties specific of free Boolean groups is given. Special emphasis is placed on 
the application of set-theoretic methods to the study of Boolean topological groups.
\end{abstract}

\keywords{Free Boolean topological group, free Boolean linear topological group, free topological 
group, free Abelian topological group, almost discrete space, Ramsey filter, extremally 
disconnected group}

\thanks{This work was supported by the Russian Foundation for Basic Research 
(project no.~15-01-05369).}

\subjclass[2010]{54H11 22A05 54G05 03E75 03E35 03C25}

\maketitle

This is a revised and expanded version of~\cite{Axioms}.

\section{Introduction}

In the very early 1940s,
A.~A.~Markov~\cite{Markov1941, Markov1945} introduced
the free topological group $F(X)$ and the free Abelian topological group $A(X)$ on 
an arbitrary completely regular Hausdorff topological space $X$ as a topological-algebraic  
counterpart of the abstract free and the  Abelian  groups on a set; he also 
proved the existence and uniqueness of these groups. 
During the next decade, Graev~\cite{Graev1948, Graev1950},
Nakayama~\cite{Nakayama}, and Kakutani~\cite{Kakutani}
simplified the proofs of the main statements of Markov's theory
of free topological groups, generalized Markov's construction, and
proved a number of important theorems on free topological groups. In particular, Graev 
generalized the notion of the free and the free Abelian topological group on a space $X$ 
by identifying the identity element of the free group with an (arbitrary) point of $X$ (the 
free topological group on $X$ in the sense of Markov coincides with Graev's group on $X$ plus an 
isolated point), described the topology of free topological groups on compact spaces, and 
extended any continuous pseudometric on $X$ to a continuous invariant 
pseudometric on $F(X)$ (and on $A(X)$) which is maximal among all such extensions~\cite{Graev1948}.

This study stimulated Mal'tsev,
who believed that the most appropriate place of the theory of abstract
free groups was in the framework of the general theory of algebraic
systems, to introduce general free topological
algebraic systems. In~1957, he published the large paper~\cite{Mal'tsev1957}, 
where the basics of the theory
of free topological universal algebras were presented.

Yet another decade later, Morris initiated the study of free topological 
groups in the most general aspect.  Namely, he introduced the notion of a variety of topological
groups\footnote{A definition of a variety of topological groups (determined by 
a so-called varietal free topological group) was also proposed in 1951 by Higman~\cite{Higman}; 
however, it is Morris' definition which has proved viable and developed into a rich theory.} 
and a full variety of topological groups and studied free objects of these
varieties \cite{Morris1, Morris2, Morris3} (see also~\cite{Morris-survey}). 
Varieties of topological groups and their free objects were also considered by
Porst~\cite{Porst}, Comfort and van~Mill~\cite{Comfort-van_Mill},
Kopperman, Mislove, Morris,
Nickolas, Pestov, and Svetlichny~\cite{Varieties-Co}, and other authors.  
Special mention should be made of Dikranjan and Tkachenko's 
detailed study of varieties of Abelian topological groups with properties related 
to compactness~\cite{Dikr-var}.

The varieties of topological groups in which free objects have been studied best are, 
naturally, the varieties of general and Abelian topological groups; free and free Abelian 
precompact groups have also been considered (see, e.g.,~\cite{book_Arh-Tkach}).
However, there is yet another natural variety---Boolean topological groups. Free objects in this variety and its subvarieties 
have been investigated much less extensively, although they arise fairly often in various studies 
(especially in the set-theoretic context). The author is aware of only two published papers 
considering free Boolean topological groups from a general point of view: \cite{Genze}, 
where the topology of the free Boolean topological group on a compact metric space was explicitly 
described, and~\cite{Genze-et-al}, where the free Boolean topological groups on compact initial 
segments of ordinals were classified (see also \cite{Genze-et-al2}).   
The purpose of this paper is to draw attention 
to these very interesting groups and give a general impression of them. We collect some 
(known and new) results on free Boolean topological groups, which describe both properties which 
these groups share with free or free Abelian topological groups 
and properties specific of free Boolean groups.  

\section{Preliminaries}
\label{Preliminaries}

All topological  spaces and groups  considered in this paper are assumed 
to be completely regular and Hausdorff. 

The notation~$\omega$ is used for the set of all nonnegative integers 
and $\mathbb N$, for the set of all positive integers. 
By $\mathbb Z_2$ we denote the group of order~2. 
The cardinality of a set $A$ is denoted by $|A|$, and the closure of a set $A$ in an 
ambient topological space is denoted by $\overline A$ and its interior, by $\Int A$. 
For any set $A$, we put 
$$
\begin{gathered}
[A]^k=\{S\subset A: |S| = k\} \quad \text{for}\ k\in \mathbb N, \\
[A]^{<\omega}=\bigcup_{k\in \mathbb N} [A]^k=\{S\subset A: |S|<\aleph_0\}, 
\quad\text{and}\quad [A]^\omega=\{S\subset A: |S|=\aleph_0\}.
\end{gathered}
$$ 
Given $s, t \in [\omega]^{<\omega}$, $s \sqsubset t$ means 
that $s$ is an initial
segment of $t$ with respect to the order induced by $\omega$. 
For $g\in [\omega]^{<\omega}\setminus \{\emptyset\}$ by $\max g$ 
we mean the greatest element of the finite set $g$ in the ordering of $\omega$. We also 
set $\max \emptyset = -1$. 

We denote the disjoint union of 
spaces $X$ and $Y$ by $X\oplus Y$. The same symbol $\oplus$ is used for direct sums of 
groups (hopefully, this will cause no confusion).

A \emph{seminorm} $\|\cdot\|$ on a group (or $\mathbb Z_2$-vector space) $G$ 
with identity element $e$ is a function $G\to \mathbb R$ such that $\|e\|=0$,  $\|g\|\ge 0$ and 
$\|g^{-1}\|=\|g\|$ for any $g\in G$, and $\|gh\|\le \|g\|+\|h\|$ for any $g, h\in G$. A seminorm 
satisfying the condition $\|g\|=0\iff g=e$ is called a \emph{norm}. 

The main object of study in this paper is Boolean topological groups. A \emph{Boolean group} is a 
group in which all elements are of order~2. Any Boolean group is Abelian: 
$xy=(yx)^2xy=yxyx^2y=yxy^2=yx$. Algebraically, all Boolean groups are free, because any Boolean 
group is a linear space over the field $\mathbb F_2=\{0,1\}$ and must have a basis (a maximal 
linearly independent set) by Zorn's lemma. This basis freely generates the given Boolean group. 
Moreover, any Boolean group (linear space) with basis $X$ is isomorphic to the direct sum 
$\bigoplus^{|X|}\mathbb Z_2$ of $|X|$ copies of $\mathbb Z_2$, i.e., the set of finitely supported 
maps $g\colon X \to \mathbb Z_2$ with pointwise addition (in the field $\mathbb F_2$). Of course, 
such an isomorphic representation depends on the choice of the basis.

Given a set or space $X$, by $F(X)$, $A(X)$, and $B(X)$ we denote, respectively, the free, free 
Abelian, and free Boolean group on $X$ with or without a topology (depending on the context). 

Topological spaces  $X$ and $Y$ are said to be \emph{$M$-equivalent} (\emph{$A$-equivalent}) if 
their free (free Abelian) topological groups are topologically isomorphic. We shall say that 
$X$ and $Y$ are \emph{$B$-equivalent} if $B(X)$ and $B(Y)$ are 
topologically isomorphic. 

Given $X\supset Y$,  we use $B(Y|X)$ to denote the topological 
subgroup  of $B(X)$ generated by $Y$. 

Whenever $X$ algebraically generates a group $G$, we set the length of the 
identity element to 0, define the length of any nonidentity $g \in G$ with 
respect to $X$ as the least (positive) integer $n$ such that $g = x_1^{\varepsilon 
_1}x_2^{\varepsilon _2}\dots x_n^{\varepsilon _n}$ for some $x_i\in X$ and $\varepsilon_i = \pm 1$, 
$i= 1, 2, \dots, n$, and denote the set of elements of length at most $k$ by $G_k$ for $k\in 
\omega$; then $G=\bigcup G_k$. Thus, we use $F_k(X)$ ($A_k(X)$, $B_k(X)$) to denote the sets of 
words of length at most $k$ in $F(X)$ (respectively, in $A(X)$ and $B(X)$). 

Let $X$ be a space, and let $X_n$, $n\in \omega$, be its subspaces 
such that $X=\bigcup X_n$. Suppose that 
any $Y\subset X$ is open in $X$ if and only if each $Y\cup X_n$ is open in $X_n$ 
(replacing  ``open'' by ``closed,'' we  obtain an equivalent condition). Then 
$X$ is said to have the \emph{inductive limit topology} 
(with respect to the decomposition $X=\bigcup X_n$). When talking about inductive limit topologies 
on $F(X)$, $A(X)$, and $B(X)$, we always mean the decompositions $F(X) = \bigcup F_k(X)$, 
$A(X) = \bigcup A_k(X)$, and $B(X) = \bigcup B_k(X)$ and always assume the sets $F_k(X)$, 
$A_k(X)$, and $B_k(X)$ to be endowed with the topology induced by the respective 
free topological groups. 

By a zero-dimensional space we mean a space $X$ with $\operatorname{ind} X=0$ and by a 
strongly zero-dimensional space, a space $X$ with $\operatorname{dim} X=0$.

\subsection*{Filters and ultrafilters}
A special place in the theory of Boolean topological groups is occupied by 
free Boolean groups on almost discrete spaces, which are closely related to filters. Recall that a 
\emph{filter} on a set $X$ is a nonempty family of susets of $X$ closed under taking finite 
intersections and supersets. A maximal (by inclusion) filter is called an \emph{ultrafilter}. A 
filter on $X$ is an ultrafilter if, given any $A\subset X$, it contains either $A$ or $X\setminus 
A$. We largely deal with filters on $\omega$. We assume all filters $\mathscr F$ on $\omega$ to be 
free, i.e., to contain the \emph{Fr\'echet} filter of all cofinite sets. 


An important role in our study is played by Ramsey, or selective, ultrafilters.

The notion of a Ramsey ultrafilter is closely related to Ramsey's theorem, which says that if 
$n \in \mathbb N$ and the set $[\omega]^n$  of $n$-element subsets of $\omega$ is partitioned 
into finitely many pieces, then there is an infinite set $H \subset \omega$ homogeneous 
with respect to this partition, i.e., such that $[H]^n$ is contained in one of 
the pieces~\cite{Ramsey}. An ultrafilter $\mathscr U$ on $\omega$ is called
 a \emph{Ramsey ultrafilter} if, given any positive integers $n$ and $k$, 
every partition $F \colon  [\omega]^n \to \{1, \dots, k\}$ has a homogeneous set 
$H \in \mathscr U$. In what follows, we use the following well-known characterizations of Ramsey 
ultrafilters.

\begin{theorem}[see \cite{Booth}]
\label{ramseychar}  
The following conditions on a free ultrafilter $\mathscr U$ on $\omega$ are equivalent:
\begin{enumerate}
\item[(i)]
$\mathscr U$ is Ramsey;
\item[(ii)]
for any partition $\{C_n: n\in \omega\}$ of $\omega$ such that $C_n\notin \mathscr U$ for 
$n\in \omega$, there exists a \emph{selector} in  $\mathscr U$, that is, a set $A\in \mathscr 
{U}$ such that $|A\cap C_n|=1$ for all~$n$; 
\item[(iii)]
for any sequence $\{A_n: n\in \omega\}$, where $A_n\in \mathscr U$, there exists an 
$A\in \mathscr U$ such  that $A=\{a_n: n\in \omega\}$ and $a_n\in A_n$ for all~$n$; 
\item[(iv)]
for any family $\{A_n: n \in\omega\}$, where $A_n\in \mathscr U$, there exists a 
\emph{diagonal intersection} in $\mathscr U$, that is, a set $D\in \mathscr U$ such that $j \in  
A_i$ whenever $i, j \in D$ and $i < j$; 
\item[(v)] for any $A_n\in \mathscr U$, $n \in\omega$, 
there exists a strictly increasing function $f\colon \omega\to\omega$ such that $f(n+1)\in 
A_{f(n)}$ for each $n\in \omega$ and the range of $f$ belongs to~$\mathscr U$. 
\end{enumerate} 
\end{theorem}

Ultrafilters with property (ii) are said to be \emph{selective}; thus, the Ramsey 
ultrafilters on $\omega$ are precisely the selective ultrafilters, and the terms ``Ramsey'' and 
``selective'' are often used interchangeably in the literature. Any Ramsey ultrafilter is 
$P$-point, but not vice versa~\cite{Booth}. 

We also mention $P$-point ultrafilters and $P$-filters. 

An ultrafilter $\mathscr U$ on $\omega$ is a \emph{$P$-point ultrafilter}, or simply a 
\emph{$P$-ultrafilter}, if, for any family of $A_i\in \mathscr U$, $i \in \omega$, 
the ultrafilter $\mathscr U$ contains 
a pseudointersection of this family, i.e., there exists an $A\in \mathscr U$  
such that $|A\setminus A_i| < \omega$ for all $i\in \omega$. The $P$-point ultrafilters 
are precisely those which are $P$-points in the remainder $\beta\omega\setminus \omega$ of the 
Stone--\v Chech compactification of the discrete space $\omega$ \cite{Booth}. By analogy, a filter 
$\mathscr F$ on $\omega$ is called a \emph{$P$-filter} if any family of $A_i\in \mathscr F$, $i \in 
\omega$, has a pseudointersection in $\mathscr F$. There are models of ZFC with 
no $P$-point ultrafilters (see \cite{Shelah}; Shelah's original proof is 
presented in~\cite{Wimmers(Shelah)}), while $P$-filters always exist: the simplest example is the 
Fr\'echet filter.  

By analogy with $P$-filters, we might define Ramsey filters as filters satisfying condition~(i) in 
Theorem~\ref{ramseychar}, but this would not yield new objects: it is easy to see that any such 
filter is an ultrafilter. The situation with selective filters is not so obvious. First, 
conditions~(ii) and~(iii), which are trivially equivalent for ultrafilters, become potentially  
different.\footnote{Moreover, interpreting $A\in \mathscr {U}$ as $\omega\setminus A\notin 
\mathscr {U}$ (this is the same thing for ultrafilters) in condition~(ii), we obtain the definition 
of $+$-selective filters~\cite{Chodounsky}, which are not necessarily ultrafilters.} Secondly, 
although the proof of the implication $\text{(iii)}\implies \text(v)$ given in~\cite{Booth} (as 
well as the trivial equivalence $\text{(iv)}\iff \text(v)$ and the obvious implication 
$\text{(iv)}\iff \text{(iii)}$) remains valid for 
filters, the standard proof of $\text{(v)}\implies \text(i)$ (see, e.g., 
\cite[Lemma~9.2]{Jech}) uses $\mathscr U$ being an ultrafilter. The author found 
several mentions (without proof) in the literature that any selective filter is an ultrafilter, but it was 
never clear from the context what exactly was meant by ``selective.'' Anyway, results 
of Section~\ref{secforcing} on free Boolean topological groups imply that any filter satisfying 
any of the equivalent conditions~(iii)--(v) is a Ramsey ultrafilter. 

Finally, a filter $\mathscr F$ on $\omega$ is said to be \emph{rapid} if every function 
$\omega\to\omega$ is majorized by the increasing enumeration of some element of $\mathscr F$. 
Clearly, any filter containing a rapid filter is rapid as well; thus, the existence of rapid 
filters is equivalent to that of rapid ultrafilters. Rapid ultrafilters are also known as 
semi-$Q$-point, or weak $Q$-point, ultrafilters. In~\cite{Miller} Miller proved that 
the nonexistence of rapid (ultra)filters is consistent with ZFC (as well that of 
$P$-point ultrafilters, as mentioned above). However, it is still unknown whether 
the nonexistence of both rapid and $P$-point ultrafilters is consistent with~ZFC. 

\section{Varieties of Topological Groups and Free Topological Groups}

A \emph{variety of topological groups} is a class of topological groups closed with
respect to taking topological subgroups, topological
quotient groups, and Cartesian products of groups with the Tychonoff product topology. 
Thus, the abstract groups $\widetilde G$ underlying the topological groups $G$ 
in a variety $V$ of topological groups (that is, all groups $G\in V$ without topology)
form a usual variety $\widetilde V$ of groups.
A variety $V$ of topological groups is \emph{full} if any topological group $G$ for which 
$\widetilde G\in \widetilde V$ belongs to $V$. The notions of a variety and a full variety 
of topological groups were introduced by Morris in~\cite{Morris1, Morris2}, who also proved 
the existence of the free group of any full variety on any completely 
regular Hausdorff space $X$.

Free objects of varieties of topological groups are
characterized by the corresponding universality properties (we give a somewhat specific meaning 
to the word ``universality,'' but we use this word only in this meaning here). Thus, the 
\emph{free topological group} $F(X)$ on a space $X$ admits the following description: 
$X$~is topologically
embedded in $F(X)$ and, for any continuous map $f$ of~$X$
to a~topological group~$G$, there exists a~unique continuous
homomorphism $\hat f\colon F(X)\to G$ for which $f= \hat f\restriction X$. As an
abstract group, $F(X)$ is the free group on the set~$X$.
The topology of $F(X)$ can be defined as the strongest group
topology inducing the initial topology on~$X$. On the other hand,
the free topological group $F(X)$ is the abstract free
group generated by the set~$X$ (which means that any map 
of the set~$X$ to any abstract group can be extended to
a homomorphism of $F(X)$) endowed with the weakest topology
with respect to which all homomorphic extensions
of continuous maps from~$X$ to topological groups are continuous. 
The \emph{free Abelian topological group} $A(X)$ on $X$, 
the \emph{free Boolean topological group} $B(X)$ on $X$, and free (free Abelian, free Boolean) 
\emph{precompact} groups are defined similarly; instead of continuous maps to
any topological groups, continuous
maps to topological Abelian groups, topological Boolean groups, and 
precompact (Abelian precompact, Boolean precompact) groups should be considered.

For any space $X$, the free Abelian topological group $A(X)$ is the quotient topological group 
of $F(X)$ by the commutator subgroup, and the free Boolean topological group 
$B(X)$ is the quotient of $A(X)$ by the subgroup of squares $A(2X)$ (which is 
generated by all words of the form $2x$, $x \in X$). (The universality of free objects 
in varieties of topological groups implies that the corresponding homomorphisms 
are continuous and open.)  
Thus, $B(X)$ is the image of $A(X)$ (and of $F(X)$) under a continuous open homomorphism. 

\subsection*{Linear topological groups}
There is yet another family of varieties of topological groups, which are not full 
but still interesting and useful. Following Malykhin (see also~\cite{book_Arh-Tkach}), 
we say that a topological group is \emph{linear} if it has a base of neighborhoods of the identity 
element which consists of open subgroups. The classes of all linear groups, all Abelian linear 
groups, and all Boolean linear groups are varieties of topological groups. As mentioned, 
these varieties are not 
full, but  for any zero-dimensional space $X$, there exist free groups 
of all of these three varieties on $X$. Indeed, a free group of a variety of topological 
groups on a given space exists if this space can be embedded as a subspace 
in a group from this variety~\cite[Theorem~2.6]{Morris1}. The following lemma ensures the 
existence of the required embeddings for the three varieties under consideration (although it 
would suffice to embed any zero-dimensional $X$ in a Boolean linear topological group, which 
belongs to eash of these varieties).

\begin{lemma} 
\label{lemma}
\begin{enumerate}
\item[(i)] For any space $X$ with $\operatorname{ind} X = 0$, there exists a 
Hausdorff linear topological group $F'(X)$ such that $F'(X)$ is the 
algebraically free group on $X$, $X$ is a closed subspace of $F'(X)$, and 
all sets $F_n(X)$ of words of length at most $n$ are closed in $F'(X)$.

\item[(ii)] For any space $X$ with $\operatorname{ind} X = 0$, there exists a 
Hausdorff Abelian linear topological group $A'(X)$ such that $A'(X)$ is the 
algebraically free Abelian group on $X$, $X$ is a closed subspace of $A'(X)$, and 
all sets $A_n(X)$ of words of length at most $n$ are closed in $A'(X)$.

\item[(iii)] For any  space $X$ with $\operatorname{ind} X = 0$, there exists a 
Hausdorff Boolean linear topological group $B'(X)$ such that $B'(X)$ is the 
algebraically free Boolean group on $X$, $X$ is a closed subspace of $B'(X)$, and 
all sets $B_n(X)$ of words of length at most $n$ are closed in $B'(X)$.
\end{enumerate}
\end{lemma}

\begin{proof} 
Assertion (i) was proved in~\cite[Theorem~10.5]{VINITI}. Let us prove (ii). 
Given a disjoint open cover $\gamma$ of $X$, 
consider the subgroup 
\begin{multline*}
H(\gamma)=\Bigl\{\sum_{i=1}^n (x_i-y_i)\colon\\ 
n\in \mathbb N\ 
\text{and for each $i\le n$, there exists an $U_i\in \gamma$ for which $x_i, y_i\in U_i$}\Bigr\}; 
\end{multline*}
clearly, we can assume that all words in $H(\gamma)$ are reduced (if $x_i$ is canceled with $y_j$, 
then $U_i=U_j$, because $U_i\cap U_j\ni x_i=y_j$ and $\gamma$ is disjoint, and we can 
replace $x_i-y_i + x_j-y_j$ by $x_j-y_i$). 
All such subgroups generate a group topology on the free Abelian group on $X$; we denote 
the free Abelian group with this topology by $A'(X)$. (We might as well take only finite 
covers.) 

The space $X$ is indeed embedded in 
$A'(X)$: given any clopen neighborhood $U$ of any point $x\in X$, we have 
$x+H(\{U, X\setminus U\})\cap X = U$. 

Let us show that $A_n(X)$ is closed in $A'(X)$ for any $n$. Take any reduced word 
$g=\varepsilon_1 x_1 + \varepsilon_2 x_2 +\dots + \varepsilon_k x_k$ with $k > n$, where 
$\varepsilon_i=\pm 1$ and $x_i\in X$ for $i\le k$. Let 
 $U_i$ be clopen neighborhoods of $x_i$ such that $U_i$ and $U_j$ are disjoint 
if $x_j\ne x_i$ and coincide if $x_j=x_j$. We set 
$$
\gamma=\Bigl\{U_1, \dots, U_k, X\setminus \bigcup_{i\le k} U_i\Bigr\}.
$$ 
Take any 
reduced word $h = \sum_{i=1}^m (y_i-z_i)$ in $H(\gamma)$ and consider $g+h$. 
If, for some $i\le m$, both $y_i$ and $-z_i$ 
are canceled in $g+h$ with some $x_j$ and $x_l$, then, first, $x_j = x_l$ (because 
any different letters in $g$ are separated by the cover $\gamma$, while 
$y_i$ and $z_i$ must belong to the same element of this cover), and secondly, $\varepsilon_j 
= -\varepsilon _l$ (because $y_i$ and $z_i$ occur in $h$ with opposite signs). Hence 
$\varepsilon_j x_j = -\varepsilon_l x_l$, which contradicts $g$ being reduced. Thus, among any two letters 
$y_i$ and $-z_i$ in $h$ only one can be canceled in $g+h$, so that $g+h$ cannot be shorter 
than $g$. In other words, $g+H(\gamma)\cap A'_n(X) =\varnothing$. 

The proof that $X$ is closed in $A'(X)$ is similar: given any 
$g\notin X$, we construct precisely the same $\gamma$ as above (if $g\notin -X$) or set 
$\gamma = \{X\}$ (if $g\in -X$) and show that 
$g+H(\gamma)$ must contain at least one negative letter.

The Hausdorffness of $A'(X)$ is equivalent to the closedness of $A_0(X)$. 

The proof of assertion~(iii) is similar.
\end{proof}

Lemma~\ref{lemma}  immediately implies the following theorem.

\begin{theorem}
\label{free-linear}
For any  space $X$ with $\operatorname{ind} X = 0$, the free, free Abelian, and free 
linear topological groups $F^{\mathrm{lin}}(X)$, $A^{\mathrm{lin}}(X)$, and 
$B^{\mathrm{lin}}(X)$ are defined. They are Hausdorff and contain $X$ as a closed subspace, 
and all sets $F_n(X)$, $A_n(X)$, and $B_n(X)$ are closed in the respective groups. 
\end{theorem}

By definition, the free linear groups of a zero-dimensional space $X$ have the strongest 
linear group topologies inducing the topology of $X$, that is,  
any continuous map from 
$X$ to a linear topological group (Abelian linear topological group, Boolean linear 
topological group) extends to a continuous homomorphism from 
$F^{\mathrm{lin}}(X)$ ($A^{\mathrm{lin}}(X)$, $B^{\mathrm{lin}}(X)$) to this group.

\section{Descriptions of the Free Boolean Group Topology}
\label{secdescriptions}

The topology of free groups can be described explicitly; all descriptions of 
the topology of free and free Abelian topological groups of which the author is aware 
are given in~\cite{VINITI}. The descriptions of the free topological group topology 
are very cumbersome (except in a few special cases); the topology of free Abelian and 
Boolean topological groups looks much simpler. Thanks to the fact that $B(X) = A(X) / A(2X)$, the 
descriptions of the free Abelian topological group topology given in~\cite{VINITI} immediately 
imply the following descriptions of the free topology of $B(X)$. 

\textbf{I\enspace} For each $n\in \mathbb N$, we fix an arbitrary
entourage $W_n\in \mathscr U$ of the diagonal of $X\times X$ in the 
universal uniformity of $X$
and set 
\begin{gather*}
\widetilde W = \{W_n\}_{n\in \mathbb N},\\
U(W_n) = \{x+y\colon
(x, y) \in W_n\},\qquad 
\\[-6pt]
\shortintertext{and}
U(\widetilde W) =
\bigcup_{n\in \mathbb N} (U(W_1) + U(W_2) +\dots + U(W_n)).
\end{gather*}
The sets $U(\widetilde W)$, where $\widetilde W$ ranges over 
all sequences of uniform entourages of the diagonal,
form a neighborhood base at zero for the topology of the free Boolean
topological group $B(X)$.

\textbf{II\enspace} For each $n\in \mathbb N$, we fix an arbitrary
normal (or merely open) cover~$\gamma_n$
of the space~$X$ and set
\begin{gather*}
\Gamma = \{\gamma_n\}_{n\in \mathbb N},\\
U(\gamma_n) = \{x + y\colon
(x, y) \in U\in \gamma_n\},
\\[-6pt]
\shortintertext{and}
U(\Gamma) =
\bigcup_{n\in \mathbb N} (U(\gamma_1) + U(\gamma_2) +\dots + U(\gamma_n)).
\end{gather*}
The sets $U(\Gamma)$, where $\Gamma$ ranges over all sequences
of normal (or arbitrary open)
covers, form a~neighborhood base at zero for the topology
of~$B(X)$.

\textbf{III\enspace} For an arbitrary continuous pseudometric~$d$ on~$X$, we set
$$
U(d)=\Bigl\{x_1+y_1+x_2+y_2+\dots+x_n+y_n\colon
n\in\mathbb N,\ x_i, y_i\in X, \ \sum_{i=1}^{n} d(x_i,y_i)<1\Bigr\}.
$$
The sets $U(d)$, where $d$ ranges over all
continuous pseudometrics on~$X$,
form a neighborhood base at zero for the topology of~$B(X)$.

\subsection*{Topology of free linear groups}

It follows directly from the second description that the base of neighborhoods of zero in 
$B^{\mathrm{lin}}(X)$ (for zero-dimensional $X$) is formed by the subgroups 
$$
\langle U(\gamma)\rangle = \Bigl\{\sum_{i=1}^n (x_i + y_i)\colon n\in \mathbb N,\
(x_i, y_i) \in U_i\in \gamma \text{ for }i\le n\Bigr\}
$$ 
generated by 
the sets $U(\gamma)$ with $\gamma$ ranging over all normal covers of $X$. By definition, 
any normal cover of a strongly dimensional space has a disjoint open refinement. Therefore, 
for $X$ with $\dim X = 0$, the covers $\gamma$ can be assumed to be disjoint, 
and for disjoint $\gamma$, 
we have
\begin{multline*}
\langle 
U(\gamma)\rangle = \Bigl\{\sum_{i=1}^n (x_i + y_i)\colon \\
n\in \mathbb N,\
(x_i, y_i) \in U_i\in \gamma \text{ for }i\le n,\ \text{the word $\sum_{i=1}^n (x_i + y_i)$ 
is reduced}\Bigr\}
\end{multline*}
(see the proof of Lemma~\ref{lemma}). 
A similar description is valid for the Abelian groups $A^{\mathrm{lin}}(X)$ (the pluses 
must be replaced by minuses). This leads to 
the following statement. 

\begin{theorem}
\label{prop-linear}
For any strongly zero-dimensional space $X$
and any $n\in \omega$, the topology induced on  $A_n(X)$ (on $B_n(X)$) by 
$A^{\mathrm{lim}}(X)$ (by $B^{\mathrm{lin}}(X)$) 
coincides with that induced by $A(X)$ (by $B(X)$).
\end{theorem}

\begin{proof}
 We can assume without loss of generality that $n$ is even. 
Given any neighborhood $U$ of zero in $A(X)$ (in $B(X)$), it suffices to take 
a sequence $\Gamma=\{\gamma_k\}_{k\in \mathbb N}$ 
of disjoint covers such that $\frac n2\cdot U(\Gamma)\subset U$ and note that 
$\langle U(\gamma_1)\rangle\cap A_n(X) \subset \frac n2\cdot U(\gamma_1)\subset U$. 
\end{proof}

\subsection*{Free topological groups in the sense of Graev and Graev's extension of pseudometrics}

In~\cite{Graev1948} Graev proposed a procedure for extending any continuous pseudometric $d$ on $X$ 
to a maximal invariant pseudometric $\hat d$ on $F(X)$, which is easy to adapt to the Boolean 
case. Following Graev, we first consider free topological groups in the sense of Graev, in which 
the identity element  is identified with a point of the generating space and the universality 
property is slightly different: only continuous maps of the generating space to a topological group 
$G$ that take the distinguished point to the identity element of $G$ must extend to continuous 
homomorphisms~\cite{Graev1948}. 
Graev showed that the free topological and Abelian topological groups 
$F_G(X)$ and $A_G(X)$ in the sense of Graev are unique (up to 
topological isomorphism) and do not depend on the choice of the distinguished point; moreover, 
the free topological group in the sense of Markov is nothing but the Graev free 
topological group on the same space to which an isolated point is added (and 
identified with the identity element). By analogy with $F_G(X)$ and $A_G(X)$ the \emph{Graev free 
Boolean topological group $B_G(X)$} can be defined: we fix a point $x\in X$, identify it with the 
zero element of $B(X)$, and endow the resulting group  with the strongest 
group topology inducing the initially existing topology on $X$, or, equivalently, with the coarsest 
group topology such that any continuous map of $X$ to a Boolean topological group $G$ that take the 
distinguished point $x$ to the zero element of $G$ extends to a continuous homomorphism. 

The topological group $B_G(X)$ thus obtained is unique (up to a topological 
isomorphism) and does not depend on the choice of the distinguished point. Indeed, let $B_{G'}(X)$ 
and $B_{G''}(X)$ be the Graev free Boolean topological groups on $X$ in which 
the zero elements are identified with $x'\in X$ and $x''\in X$, respectively. The map 
$\varphi\colon X\to B_{G''}(X)$ taking each point of $X\subset B_{G'}(X)$ to the point $x+x' \in 
B_{G''}(X)$ is continuous, and the image of $x'$ is the zero of $B_{G''}(X)$. Therefore, $\varphi$ 
can be extended to a continuous homomorphism $\tilde \varphi\colon B_{G'}(X) \to B_{G''}(X)$. 
Similarly, the map $\psi\colon X\to B_{G'}(X)$ taking each point of $X\subset B_{G''}(X)$ to the 
point $x+x'' \in B_{G'}(X)$ is continuous, and the image of $x''$ is the zero of $B_{G'}(X)$. 
Therefore, $\psi$ can be extended to a continuous homomorphism $\tilde \psi\colon B_{G''}(X) \to 
B_{G'}(X)$. For each point $x\in X\subset B_{G'}(X)$, we have $$ \tilde\psi\tilde\varphi(x) = 
\tilde\psi(x+x')=\tilde\psi(x)+\tilde\psi(x')= x+x''+x'+x''=x+x'=x, $$ because $x'$ is the zero 
element of $B_{G'}(X)$. Thus, the continuous self-homomorphism $\tilde\psi\tilde\varphi$ of 
$B_{G'}(X)$ is the identity map on $X$. Since $X$ generates the group $B_{G'}(X)$, it follows that 
$\tilde\psi\tilde\varphi$ is the identity automorphism of $B_{G'}(X)$, and hence $\tilde\varphi$ is 
an isomorphism of $B_{G'}(X)$ onto $B_{G''}(X)$.

The extension  of a continuous pseudometric $d$ on $X$ 
to a maximal invariant continuous pseudometric $\hat d$ on 
the Graev free Boolean topological group $B_G(X)$  
is defined by setting 
$$
\hat d(g, h)=
\inf\Bigl\{\,\sum_{i=1}^n d(x_i, y_i)\colon n\in \mathbb N,\ x_i, y_i\in X, 
g= \sum_{i=1}^n x_i, \ h= \sum_{i=1}^n y_i\Bigr\}
$$
for any $g, h\in B_G(X)$. 
The infimum is taken over all representations of $g$ and $h$ as (reducible) words of 
equal lengths. 
The corresponding Graev seminorm 
$\|\cdot\|_d$ (defined by $\|g\|_d = \hat d(g, 0)$ for $g\in B_G(X)$, where 0 is the zero 
element of $B_G(X)$) is given by 
$$
\|g\|_d = 
\inf\Bigl\{\sum_{i=1}^n d(x_i, y_i)\colon
g= \sum_{i=1}^n (x_i+y_i),\ x_i, y_i\in X\Bigr\}.
$$
The infimum is attained at a word representing $g$ which may contain one 0 (if the 
length of $g$ is odd) and is otherwise reduced. Indeed, 
if the sum representing $g$ contains terms of the form $x+z$ and $z+y$,
then these terms can be replaced by one term $x+y$; the sum
$\sum_{i=1}^n d(x_i, y_i)$ does not increase under such a change thanks to the 
triangle inequality.

For the usual (Markov's) free Boolean topological group $B(X)$, 
which is the same as $B_G(X\oplus\{0\})$ (where $0$ is an isolated point idetified with  
zero), the Graev metric depends on the distances from the points of $X$ to the isolated point 
(they are usually set to 1 for all $x\in X$). The  
corresponding seminorm $\|\cdot\|_d$ on the subgroup $B_{\mathrm{even}}(X)$ of 
$B(X)$ consisting of words of  
even length does not change. 
The subgroup $B_{\mathrm{even}}(X)$ is open and closed in $B(X)$, 
because this is the kernel of the continuous homomorphism $\hat f\colon B(X)\to \{0, 1\}$ extending
the constant continuous map $f\colon X\to \{0, 1\}$ taking all $x\in X$ to $1$. Thus, in fact, 
it does not matter how to extend $\|\cdot\|_d$ to $B(X)\setminus B_{\mathrm{even}}(X)$; 
for convenience, we set 
$$ 
\|g\|_d= \begin{cases}
\min\Bigl\{\vtop{\hbox{\vphantom{$\Bigl\{$} $\sum_{i=1}^n d(x_i, y_i)\colon$} 
                 \hbox{\strut\quad $g= \sum_{i=1}^n (x_i+y_i),\ x_i, y_i\in X$,}
                 \hbox{\quad the word $\sum_{i=1}^n (x_i+y_i)$ is reduced$\Bigr\}$}}
& \text{ if $g\in B_{\mathrm{even}}(X)$},\\
1&\text{ if $g\in B(X)\setminus B_{\mathrm{even}}(X)$}.
\end{cases}
$$
All open balls of radius $1$ (or of any radius smaller than $1$) in all seminorms 
$\|\cdot\|_d$ for $d$ ranging over all continuous pseudometrics on $X$ form a base 
of open neighborhoods of zero in $B(X)$.

\subsection*{Boolean groups generated by almost discrete spaces}

A special role in the theory of topological groups and in set-theoretic topology is played 
by Boolean topological groups generated by \emph{almost discrete} spaces, that is, spaces having 
only one nonisolated point. Clearly, for any almost discrete space $X \cup \{*\}$, where $*$ is the 
only nonisolated point, the punctured neighborhoods of $*$ form a filter on $X$. Conversely, 
each free filter $\mathscr F$ on any set $X$ is naturally associated with the almost discrete space 
$X_{\mathscr F} = X \cup \{*\}$ ($*$ is a point not belonging to $X$); all points of $X$ are 
isolated and the neighborhoods of $*$ are $\{*\}\cup A$, $A\in \mathscr F$. For a space $X$ with 
infinitely many isolated points, there is no difference between the canonical definition of the 
groups $F(X)$, $A(X)$, and $B(X)$ and Graev's generalizations $F_G(X)$, $A_G(X)$, and $B_G(X)$: as 
mentioned above, Markov's group $B(X)$ is topologically isomorphic to $B_G(Y)$, where $Y$ is $X$ 
plus an extra isolated point. Thus, when dealing with spaces $X_{\mathscr F}$ associated with 
filters, we can identify $B(X_{\mathscr F})$ with $B_G(X_{\mathscr F})$ and assume that  the only 
nonisolated point is the zero of $B(X_{\mathscr F})$; the descriptions of the neighborhoods of zero 
and the Graev seminorm are altered accordingly. To understand how they change, take the new (but in 
fact the same) space $\widetilde X_{\mathscr F}= X_{\mathscr F}\cup\{0\}$, where $0$ is one more 
isolated point, represent $B(X_{\mathscr F})$ as the Graev free Boolean topological group 
$B_G(\widetilde X_{\mathscr F})$ with distinguished point (zero of $B_G(\widetilde X_{\mathscr 
F})$) $0$, and consider the topological isomorphism $g\mapsto g+0$ between this group and the 
similar group with distinguished point (zero)~$*$.

For example, since 
any open cover of $X_{\mathscr F}$ can be assumed to consist of a neighborhood of $*$ 
and  singletons, the description~II reads as follows in this 
case: For each $n\in \mathbb N$, we fix an arbitrary
neighborhood~$V_n$ of $*$, that is, $A_n\cup\{*\}$, where $A_n\in \mathscr F$, 
and set
$W = \{V_n\}_{n\in \mathbb N}$, $U(V_n) = \{x\colon
x\in V_n\}=\{x\colon x\in A_n\}$ ($*$ is zero in $B_G(X)$), and 
$$
U(W) =
\bigcup_{n\in \mathbb N} (U(V_1) + U(V_2) +\dots + U(V_n)) = 
\bigcup_{n\in \mathbb N}\{x_1+ \dots + x_n\colon x_i\in A_i\text{ for }i\le n\}.
$$
The sets $U(W)$, where the $W$ range over all sequences
of neighborhoods of $*$, form a~neighborhood base at zero for the topology
of $B(X_{\mathscr F})$. Strictly speaking, to obtain a full analogy with the description~II  
of the Markov free group topology, we should set 
\begin{multline*}
U(W) =
\bigcup_{n\in \mathbb N} (2U(V_1) + 2U(V_2) +\dots + 2U(V_n))\\ 
= \bigcup_{n\in \mathbb N}\{x_1+y_1 \dots + x_n+y_n\colon x_i,y_i\in V_i\text{ for }i\le n\},
\end{multline*}
but this would not affect the topology: the former $U(W)$ equals the latter for 
a sequence of smaller neighborhoods, say 
$V'_n = \bigcap_{i\le 2n}V_{i}$ (remember that some of the $x_i$ and $y_i$ in the 
expression for $U(W)$
may equal $*$, that is, vanish).

Similarly, the base neighborhoods of zero in description~III take the form 
$$
U(d)=\Bigl\{x_1+x_2+\dots+x_n\colon
n\in\mathbb N,\ x_1, \dots, x_n\in X,\ \sum_{i=1}^{n} d(x_i,*)<1\Bigr\},
$$
where $d$ ranges over continuous pseudometrics on $X_{\mathscr F}$. (Again, we should 
set 
\begin{multline*}
U(d)=\Bigl\{x_1+y_1+ x_2+y_2+\dots+x_n+y_n\colon\\
n\in\mathbb N,\ x_i, y_i\in X\cup\{*\},\ \sum_{i=1}^{n} d(x_i,y_i)<1\Bigr\},
\end{multline*}
but this would not make any difference.)

It is also easy to see that the isomorphism between $B_G(\widetilde X_{\mathscr F})$ (with 
distinguished point $*$) and $B(X_{\mathscr F})$ does not essentially affect the sets 
of words of length at most $n$; in particular, they remain closed, and 
$B_G(\widetilde X_{\mathscr F})$ is the inductive limit 
of these sets with the induced topology 
if and only if $B(X_{\mathscr F})$ has the inductive limit topology. In what 
follows, we identify $B_G(\widetilde X_{\mathscr F})$ with $B(X_{\mathscr F})$ and use the notation 
$B(X_{\mathscr F})$ for the Graev free Boolean topological group on $X_{\mathscr F}$ 
with zero $*$. 

Thus,  $B(X_{\mathscr F})$  
is naturally identified with the set $[X]^{<\omega}$ of all finite subsets 
of $X$ with the operation $\triangle$ of symmetric difference ($A \triangle B = 
(A\setminus B)\cup (B\setminus A)$). The point $*$, 
which is the zero element of $B(X_{\mathscr F})$, is identified with 
the empty set $\emptyset$, which belongs to $[X]^{<\omega}$ as the zero element. 
Sets of the form $[X]^{<\omega}$ often arise in set-theoretic topology 
and in forcing. The role of $X$ is usually played by $\omega$, and the filter $\mathscr F$ 
is often an ultrafilter with certain properties. 

In the context 
of free Boolean groups on almost discrete 
spaces we identify each $n\in \omega$ with the one-point 
set $\{n\}\in [\omega]^{<\omega}$.

\section{A Comparison of Free, Free Abelian,\\ and Free Boolean Topological Groups}

\subsection*{Similarity}
There are a number of known properties of free and free Abelian topological groups which 
automatically carry over to free Boolean topological groups simply because they are preserved by 
taking topological quotient groups or, more generally, by continuous maps. 
Thus, if $F(X)$ (and $A(X)$) is separable, Lindel\"of, ccc, and so on, 
then so is $B(X)$. It is also quite obvious that $X$ is discrete if and only if so are 
$F(X)$, $A(X)$, and $B(X)$. 

Let $X$ be a space, and let $Y$  be its subspace. The topological 
subgroup $B(Y|X)$ of $B(X)$ generated by $Y$ 
is not always the free Boolean topological group on $Y$ (the induced topology may be 
coarser). Looking at the description~I of the free group topology on $B(X)$, we see that $X$ 
and $Y$ equipped with universal uniformities $\mathscr U_X$ and $\mathscr U_Y$ are uniform 
subspaces of $B(X)$ and $B(Y)$ with their group uniformities $\mathscr W_{B(X)}$ and $\mathscr 
W_{B(Y)}$ (generated by entourages of the form $W(U) = \{(g, h)\colon h\in g + U\}$, where $U$ 
ranges over all neighborhoods of zero in the corresponding group), which completely determine the 
topologies of $B(X)$ and $B(Y)$. Thus, if the topology of $B(Y|X)$ coincides with that of 
$B(Y)$, then, like in the case of free and free Abelian topological 
groups~\cite{Uspenskii1990, Tkachenko-Abelian},  
$(Y, \mathscr U_Y)$ must be a uniform subspace of $(X, \mathscr U_X)$, which means that any 
bounded continuous pseudometric on $Y$ can be extended to a continuous pseudometric on $X$ 
(in this case, $Y$ is said to be \emph{$P$-embedded} in $X$~\cite{P-embedding}). 
The converse has been proved to be true for 
free Abelian (announced in~\cite{Tkachenko-Abelian}, proved in~\cite{Uspenskii1990}) 
and even free \cite{Sipa-subspaces}\footnote{See also~\cite{VINITI}, where a minor misprint
in the condition~$3^\circ$ on p.~186 of~\cite{Sipa-subspaces} is corrected.}
topological groups. This immediately implies the following 
theorem.

\begin{theorem}
\label{subspace}
Let $X$ be a space, and let $Y$ be its subspace. The topological subgroup of the 
free Boolean groups $B(X)$ generated by $Y$ is the free topological group $B(Y)$ if
and only if each bounded continuous pseudometric on $Y$ can be
extended to a continuous pseudometric on $X$.
\end{theorem}

Any space $X$ is closed in its free Boolean topological group $B(X)$, as well as in $F(X)$ 
and $A(X)$ (see, e.g., \cite[Theorems~2.1 and 2.2]{Morris2}). Moreover, all $F_n(X)$, $A_n(X)$, 
and  $B_n(X)$ (the sets of words of length at most $n$) are closed in their respective groups as 
well. The most elegant proof of this fact was first proposed by Arkhangel'skii in the unavailable 
book~\cite{Arkhangel'skii1969} (for $F(X)$, but the argument works for $A(X)$ and $B(X)$ without 
any changes): note that all $F_n(\beta X)\subset F(\beta X)$ are compact, since these are 
the continuous images of $(X\oplus \{e\} \oplus X^{-1})^n$ 
under the natural multiplication maps 
$i_n\colon (x_1^{\varepsilon_1}, \dots , x_n^{\varepsilon_n})\mapsto x_1^{\varepsilon_1} \dots x_n^{\varepsilon_n}$. 
(Here $e$ denotes the identity element of $F(X)$, $\varepsilon_i = \pm 1$, and the word 
$x_1^{\varepsilon_1} \dots x_n^{\varepsilon_n}$ may be reducible, i.e., have length 
shorter than $n$.)
Therefore, the $F_n(\beta X)$ are closed in $F(\beta X)$, and hence the sets $F_n(X)=F_n(\beta X) 
\cap F(X|\beta X)$ are closed in $F(X|\beta X)$. It follows that these sets are also closed 
in $F(X)$, which is the same group as $F(X|\beta X)$ but has stronger topology. 

The topological structure of a free group becomes much clearer when this group 
has the inductive limit topology (or, equivalently, when the inductive limit topology 
is a group topology). The problem of describing all spaces for which $F(X)$ (or $A(X)$) possesses  
this property has proved extremely difficult (and is still unsolved). 
Apparently, the problem was first stated explicitly by Pestov and
Tkachenko in 1985~\cite{UP}, but it was tackled as early as in 1948 by Graev~\cite{Graev1948}, 
who proved that the free topological group of a compact space has the
inductive limit topology. Then Mack, Morris, and Ordman~\cite{MMO} proved
the same for $k_{\omega}$-spaces. Apparently, the strongest
result in this direction was obtained by Tkachenko~\cite{Tkachenko-indlim}, 
who proved that if
$X$ is a $P$-space or a $C_{\omega}$-space (the latter means $X$
is the inductive limit of an increasing sequence $\{X_{n}\}$ of its closed
subsets such that all finite powers of each $X_{n}$ are countably compact
and strictly collectionwise normal), then $F(X)$ has the inductive limit
topology. All these sufficient conditions are also 
valid for $A(X)$ and $B(X)$ by virtue of the following 
simple observation.  

\begin{proposition} 
\label{indlim-prop}
Suppose that $X=\bigcup_{n\in \omega} X_n$, $Y =\bigcup_{n\in \omega} Y_n$, 
$X$ is the inductive limit of its subspaces $X_n$, $n\in \omega$, and 
$f\colon X \to Y$ is an open continuous map such that $f(X_n)=Y_n$ for each $n\in \omega$. 
Then $Y$ is the inductive limit of its subspaces $Y_n$. 
\end{proposition}

\begin{proof} Let $U\subset Y$ be such that all 
$U_{n}=U\cap Y_n$ are open in $Y_n$. 
Consider $V=f^{-1}(U)$ and
$V_{n}=f^{-1}(U_{n})\cap X_{n}$ for $n\in \omega$.  For each $n$, fix open
$W_{n}\subset Y$  for which $W_{n}\cap Y_{n}=U_{n}$. We have
$$
V_{n}=f^{-1}(W_{n}\cap Y_{n})\cap X_{n}
= (f^{-1}(W_{n})\cap f^{-1}(Y_{n}))\cap X_{n}=f^{-1}(W_{n})\cap X_{n};
$$
hence each $V_{n}$ is open in $X_{n}$.
On the other hand,
$$
V_{n}=f^{-1}(U\cap Y_{n})\cap X_{n}=
(f^{-1}(U)\cap f^{-1}(Y_{n}))\cap X_{n}=
V\cap X_{n};
$$
therefore, $V$ is open in $X$. Since the map $f$ is open, it follows that 
$U=f(V)$ is an open set.
\end{proof}

For $X$ of the form $\omega_{\mathscr F}$ (where $\mathscr F$ is a filter on $\omega$), not only 
the sufficient conditions mentioned above 
but also a necessary and sufficient condition for $F(X)$ and $A(X)$ to have the 
inductive limit topology is known. This condition is also valid for $B(X)$. 

\begin{theorem}
Given a filter $\mathscr F$ on $\omega$, $B(\omega_{\mathscr F})$ has the inductive limit topology 
if and only if $\mathscr F$ is a $P$-filter. 
\end{theorem}

\begin{proof}
This theorem is true for free and free Abelian topological groups~\cite{Sipa-indlim}. 
Therefore, by Proposition~\ref{indlim-prop}, 
$B(\omega_{\mathscr F})$ has the inductive limit topology for any $P$-filter. It remains to prove 
that if $\mathscr F$ is not a $P$-filter, then $B(\omega_{\mathscr F})$ is not the inductive 
limit of the $B_n(\omega_{\mathscr F})$.

Thus, suppose that $\mathscr F$ is not a $P$-filter (or, equivalently, there 
exist a decreasing sequence of $A_n\in \mathscr F$, $n\in \omega$, 
such that, for any $A\in \mathscr F$, there is an $i$ for which the 
intersection $A\cap A_i$ is infinite) but $B(\omega_{\mathscr F})$ is the inductive limit
of the $B_n(\omega_{\mathscr F})$. As usual, we assume that the zero element of 
$B(\omega_{\mathscr F})$ is 
the nonisolated point $*$ of $\omega_{\mathscr F}$. 

Without loss of generality, we can assume that $A_0=\omega$ and all sets 
$A_n\setminus A_{n+1}$ are infinite. We enumerate these sets as 
$$
A_n\setminus A_{n+1}=\{x_{ni}\colon i\in \omega\}
$$
and put 
\begin{multline*}
	D_{n}=\{x_{nm}+x_{i_1j_1}+x_{i_2j_2}+ \dots +x_{i_nj_n}\colon \\
n<i_1<i_2<\dots <i_n<j_1< j_2< \dots <j_n<m\}
\end{multline*}
for all $n\in\omega$.
Let us show that each $D_{n}$ is a closed discrete subset of $B(\omega_{\mathscr F})$. 
Fix $n$
and consider $X=\{*\}\cup\{x_{nm}\colon  m\in\omega\}$ and the retraction $r\colon 
\omega_{\mathscr F}\to X$ 
that maps $\omega_{\mathscr F}\setminus X$ to $\{*\}$. Clearly, $X$ is discrete and the map
$r$ is continuous. Let $\hat r\colon B(\omega_{\mathscr F})\to B(X)$
be the homomorphic extension of $r$; then $\hat r$
continuously maps $B(\omega_{\mathscr F})$ onto the discrete group $B(X)$.
For any $g\in B(\omega_{\mathscr F})$, the set $\hat r^{-1}(g)\cap
D_{n}$ is finite: if $\hat r^{-1}(g)\cap
D_{n}$ is nonempty, then we have 
$g= \hat r(x_{nm_{0}}+x_{i_{0_1}j_{0_1}}+x_{i_{0_2}j_{0_2}}+ \dots +
x_{i_{0_n}j_{0_n}})$ 
for some $m_{0},i_{0_k},j_{0_k}\in
\omega$ such that $n<i_{0_1}<i_{0_2}<\dots <i_{0_n}<j_{0_1}< j_{0_2}< \dots <j_{0_n}<m$, whence 
$g= x_{nm_{0}}$ 
and
\begin{multline*}
\hat r^{-1}(g)\cap D_{n}=
\{x_{nm_{0}}+x_{i_1j_1}+x_{i_2j_2}+ \dots +x_{i_nj_n}
\colon\\ 
n<i_1<i_2<\dots <i_m<j_1< j_2< \dots <j_n<m_{0}\}.
\end{multline*}
Since the sets $\hat r^{-1}(g)$, $g\in B(X)$, form an open cover of $B(\omega_{\mathscr F})$, 
it follows that $D_{n}$ is a closed discrete subspace of $B(\omega_{\mathscr F})$.

The length of each word in $D_n$ equals $n+1$. 
Therefore, $D=\bigcup_{n}D_n$ is closed in the inductive limit topology.
It remains to show that $*$ (the zero of $B(\omega_{\mathscr F})$) 
belongs to the closure of $D$ in the free group 
topology, i.e., that $U(d)\cap D\neq\emptyset$ for any continuous pseudometric
$d$ on $\omega_{\mathscr F}$ (see the description~III of the topology of $B(\omega_{\mathscr F})$).

Take an arbitrary (continuous) pseudometric $d$ on $\omega_{\mathscr F}$. 
In $\omega_{\mathscr F}$ 
the ball $B_{d}(*,\frac12)$ of radius $\frac12$ centered at $*$ with respect to $d$ 
is a neighborhood of $*$; that is, the punctured ball (with $*$ removed)
belongs to $\mathscr F$. By assumption, the set
$M=\{m\in\omega:\ d(*,x_{nm})<\frac12)\}$ is infinite for some
$n\in\omega$. Since 
$B_{d}(*,\frac1{2n})\cap A_{n+1}$ is a punctured neighborhood of $*$,
it follows that the sets 
$J_i=\{j\in\omega:\ d(*,x_{ij})<\frac1{2n}\}$ are 
infinite for infinitely many $i>n$. Choose $i_1<i_2 <\dots <i_n$ greater than $n$ so that 
all $J_{i_k}$ are infinite, in each $J_{i_k}$ choose $j_k$ so that $i_n<j_1<\dots<j_n$, 
and take $m\in M$ such that $m>j_n$. We have 
$g=x_{nm}+x_{i_1j_1}+x_{i_2j_2}+ \dots +x_{i_nj_n}\in D_{n}$. 
We also have $g\in U(d)$, because
$$
d(*,x_{nm})+\sum_{k=1}^n d(*,x_{i_kj_k}) < \frac12 +
n\frac1{2n}=1.
$$
Therefore, $g \in D_{n}\cap U(d)$.
\end{proof}

In~\cite{Tkachuk-oplus} Tkachuk proved that the free Abelian topological group of a 
disjoint union of two spaces $X$ and $Y$ is topologically isomorphic to  the direct sum 
$A(X)\bigoplus A(Y)= A(X)\times A(Y)$. His argument carries over  
to varieties of Abelian topological groups closed under direct sums 
(or, in topological terminology, 
$\sigma$-products with respect to the zero elements of factors) 
with the box topology. We denote such sums by $\sigma\square$.

\begin{theorem}
\label{oplus}
For any family $\{X_\alpha\colon \alpha\in A\}$ of spaces, 
$$
A\Bigl(\bigoplus_{\alpha\in A} X_\alpha\Bigr) \cong
\sigma\square_{\alpha\in A} A(X_\alpha)
\qquad\text{and}\qquad
B\Bigl(\bigoplus_{\alpha\in A} X_\alpha\Bigr) \cong
\sigma\square_{\alpha\in A} B(X_\alpha).
$$ 
If all $X_\alpha$ are zero-dimensional, then 
$$
A^{\mathrm{lin}}\Bigl(\bigoplus_{\alpha\in A} X_\alpha\Bigr) \cong
\sigma\square_{\alpha\in A} A^{\mathrm{lin}}(X_\alpha)
\qquad\text{and}
\qquad 
B^{\mathrm{lin}}\Bigl(\bigoplus_{\alpha\in A} X_\alpha\Bigr) \cong
\sigma\square_{\alpha\in A} B^{\mathrm{lin}}(X_\alpha).
$$
\end{theorem}

\begin{proof}
Let $T$ stand for $A$, $B$, $A^{\mathrm{lim}}$, or $B^{\mathrm{lim}}$, and let 
$0_\alpha$ denote the zero element of $T(X_\alpha)$.
For each $\alpha\in A$, we set $X'_\alpha = \sigma\square_{\beta\in A} Y_\beta$, where 
$Y_\alpha = X_\alpha$ and $Y_\beta=\{0_\beta\}$ for $\beta\ne\alpha$. Every $X'_\alpha$ is 
embedded in the group $T'_\alpha(X_\alpha)$ defined accordingly as 
a product of $T(X_\alpha)$ and zeros.  Clearly, the union 
$\bigcup_{\alpha\in A} X'_{\alpha}$ algebraically generates 
$\sigma\square_{\alpha\in A} T(X_\alpha)$ and 
is homeomorphic to $\bigoplus_{\alpha \in A}X_\alpha$. It remains to show that the 
homomophic extension of any 
continuous map of this union to any topological group from the corresonding variety 
is continuous. Let $f\colon \bigcup_{\alpha\in A} X'_{\alpha} \to G$ be such a map. 
For each $\alpha\in A$, the homomorphic extension $\hat f_\alpha\colon T'_\alpha(X_\alpha)\to 
G$ of the restriction of $f$ to $X'_\alpha$ is continuous. We define 
$\hat f\colon \sigma\square_{\alpha\in A} T(X_\alpha) \to G$ by setting 
$\hat f\bigl((g_\alpha)_{\alpha\in A}\bigr) = \sum_{\alpha \in A} \hat f_\alpha (g_\alpha)$ 
for each $(g_\alpha)_{\alpha\in A}\in \sigma\square_{\alpha\in A} T(X_\alpha)$; the sum is defined, 
because any element of $\sigma\square_{\alpha\in A} T(X_\alpha)$ has only 
finitely many nonzero components. 
Let us show that $\hat f$ is continuous. It suffices to 
check continuity at the zero element of $\sigma\square_{\alpha\in A} T(X_\alpha)$. 
Take any neighborhood $U$ of zero in $G$. Its preimages $V_\alpha$ under the component maps 
$\hat f_\alpha$ are open neighborhoods of zero in $T'_\alpha(X_\alpha)$. The product 
$\sigma\square_\alpha V_\alpha$ is the preimage of $U$ under $\hat f$, and it is open in the box 
topology. 
\end{proof}

The free Boolean topological group of a nondiscrete space is never metrizable (as well as the 
free and free Abelian topological groups). Indeed, if 
$B(X)$ is metrizable and $X$ is nondiscrete, then $X$ contains a convergent sequence 
$S$ with limit point $*$, and $B(S) = B(S|X)$ 
(see Theorem~\ref{subspace}); thus, it suffices to show that $B(S)$ is nonmetrizable. Suppose that 
it is metrizable. Then the topology of $B(S)$ is generated by a continuous norm 
$\|\cdot\|$. For all pairs of positive integers $n$ and $m\le n$ choose  
different $s_{n_m}\in S$ so that 
$\|s_{n_m}+*\|< \frac{1}{n^2}$. Clearly, the set 
$$
D = \{(s_{n_1}+*) + (s_{n_2}+*) + \dots (s_{n_n}+*)\colon n \ge 0\}
$$
has finite intersection with each $B_k(S)$; hence it must be discrete, because $B(S)$ has the 
inductive limit topology. On the other hand, $D$ is a sequence convergent to zero, since 
$$
\|(s_{n_1}+*) + (s_{n_2}+*) + \dots (s_{n_n}+*)\|\le \sum_{i=1}^n (s_{n_i}+*) < 
n\cdot \frac{1}{n^2} = \frac 1n.
$$

The list of properties shared by free, free Abelian, and free Boolean topological groups that can 
be proved without much effort is very long. Many of these properties are proved for Boolean groups 
by analogy, but sometimes their proofs are drastically simplified. We conclude our brief excursion 
by one of such examples. The proof of the following theorem for free topological groups given 
in~\cite{Sipa-subspaces} is extremely complicated (it is based on a more general construction). The 
proof given in~\cite{VINITI} is much shorter but still very cumbersome. In the Boolean case, the 
proof becomes almost trivial.

\begin{theorem}
If $\dim X = 0$, then $\operatorname{ind} B(X) = 0$.
\end{theorem}

\begin{proof}
Any continuous pseudometric $d$ on $X$ is majorized by a non-Archimedean 
pseudometric\footnote{A pseudometric $\rho$ is said to be \emph{non-Archimedean} if 
$\rho(x, z)\le \max\{\rho(x, y),\allowbreak \rho(y, z)\}$ for any $x, y, z\in X$.}
$\rho$ 
taking only values of the form $\frac1{2^n}$. To see this, it 
suffices to consider the elements $V_0, V_1, \dots$ of the universal uniformity on $X$ which 
are determined by decreasing disjoint open refinements $\gamma_0, \gamma_1, \dots$ 
of the covers of $X$ by balls of radii $\frac1{2^1}, \frac1{2^2},\dots$ with respect to $d$ and 
apply the construction in the proof of Theorem~8.1.10 of~\cite{Engelking} 
(see also~\cite{de_Groot}). Since the covers 
$\gamma_n$ determining the entourages $V_n$ are disjoint and each $\gamma _{i+1}$ is a refinement 
of $\gamma_i$, it follows that the function $f$ in this construction has the property 
$f(x, z)\le \max \{f(x, y), f(y, z)\}$ and, therefore, the pseudometric $\rho$ constructed there 
from $f$ is non-Archimedean and takes the values $\frac1{2^n}$. Clearly, it majorizes $d$.

Each value $\|g\|_\rho$, $g\in B(X)$, of the Graev extension $\|\cdot\|_\rho$ of $\rho$ 
is either 1 or a finite sum of values of $d$ (recall that the 
minimum in the expression for $\|g\|_\rho$ is attained at the irreducible representation of $g$). 
Hence $\|\cdot\|_\rho$ takes only rational values, and the balls with irrational radii 
centered at zero in this norm are open and closed. They form a base of neighborhoods of zero, 
and their translates, a base of the entire topology on $B(X)$.
\end{proof}

\subsection*{Difference}
Pestov gave an example of a space $X$ for which $F(X)$ is not homeomorphic to 
$A(X)$~\cite{Pestov}. Spaces for which $A(X)$ is not homeomorphic to $B(X)$ exist, too. 

\begin{proposition}
The free Abelian topological group of any connected space has infinitely many 
connected components. The free Boolean topological group of any connected space 
has two connected components.
\end{proposition}

\begin{proof}
Consider a connected space $X$. The connected component of zero in $A(X)$ is the subgroup 
$A^{\mathrm c}(X)$ consisting of all words $\sum_{i=1}^n x_i^{\varepsilon_i}$ 
with $\sum_{i=1}^n \varepsilon_i = 1$ (see~\cite[Lemma~7.10.2]{book_Arh-Tkach}). 
Clearly, all words in this subgroup are of even length, 
and the  canonical homomorphism $A(X)\to B(X)$ takes $A^{\mathrm c}(X)$ to the 
subgroup $B^{\mathrm c}(X)$ of $B(X)$ consisting of all words of even length. Since 
the canonical homomorphism is continuous, the subgroup $B^{\mathrm c}(X)$ is connected, and it has 
index $2$ in $B(X)$. Thus, $B(X)$ has at most two (in fact, precisely two) connected components, 
while $A(X)$ has infinitely many connected components, because $A(X)/A^{\mathrm c}(X)\cong \mathbb 
Z$ (see \cite[Lemma~7.10.2]{book_Arh-Tkach}). 
\end{proof}

There is a fundamental difference in the very topological-algebraic nature of free, free Abelian, 
and free Boolean groups. Thus, nontrivial free and free Abelian groups admit 
no compact group topologies 
(see \cite{Tkachenko-pseudocompact}); this follows from the well-known algebraic description of 
infinite compact Abelian groups \cite[Theorem~25.25]{H-R}. On the other hand, for any infinite 
cardinal $\kappa$, the direct sum $\bigoplus_{2^\kappa}\mathbb Z_2$ of $2^\kappa$ copies of 
$\mathbb Z_2$ (that is, the free Boolean group of rank 
$2^\kappa$) is algebraically  isomorphic to the Cartesian product $(\mathbb Z_2)^\kappa$ 
\cite[Lemma~4.5]{H-R1} and, therefore, admits compact group topologies 
(e.g., the product topology). 

The free and free Abelian groups are never finite, while the free Boolean group of any finite set 
is finite.

The free and the free Abelian topological group of any completely regular Hausdorff topological 
space $X$ contain all finite powers $X^n$ of $X$ as closed subspaces. Thus, each $X^n$ is 
homeomorphic to the closed subset $\{x_1\dots x_n\colon x_i\in X \text{ for }i=1\le n\}$ of $F(X)$ 
(see~\cite{Arh1968}\footnote{Arkhangel'skii announced this result in~\cite{Arh1968} and proved 
it in~\cite{Arkhangel'skii1969} by considering the Stone--\v Cech compactification 
$\beta X$ of $X$ and its 
free topological group; details can be found in \cite[Theorem~7.1.13]{book_Arh-Tkach}. 
Unfortunately, the book~\cite{Arkhangel'skii1969}, which is a rotaprint edition of a lecture course, is (and 
always was) virtually unavailable, even in Russia. Thus, the result was rediscovered by 
Joiner~\cite{Joiner} and the idea of proof, by Morris~\cite{Morris2} (see 
also~\cite{HMT}). In fact, both Arkhangel'skii and Joiner proved a stronger statement; namely, they 
gave the same complete description of the topological structure of all $F_n(X)$, although obtained 
by different methods (Arkhangel'skii proof is much shorter).}) and to the closed subset 
$\{x_1+2x_2 + \dots + nx_n\colon x_i \in X \text{ for }i=1\le n\}$ of $A(X)$~\cite{Thomas}. 
However, the situation with free Boolean topological groups is much more complicated. For example, 
consider extremally disconnected free topological groups. 

Extremally disconnected groups are discussed in the next section. Here we only mention that 
nondiscrete $F(X)$ and $A(X)$ are never extremally disconnected, while  $B(X)$ may be 
nondiscrete and extremally disconnected under certain set-theoretical assumptions 
(e.g., under CH) even for countable $X$ of the form $\omega_{\mathscr F}$, and that any  
hereditarily normal, in particular, countable, extremally disconnected  space is hereditarily 
extremally disconnected (this is shown in the next section). It follows that if $X$ is a 
nondiscrete countable space for which $B(X)$ is extremally disconnected, then $B(X)$ does not 
contain $X^2$ as a subspace. Indeed, otherwise, $X^2$ is extremally disconnected (and nondiscrete), 
and the existence of such spaces is prohibited by the following simple observation; it must be 
known, although the author failed to find a reference. 

\begin{proposition}
If $X\times X$ is extremally disconnected, then $X$ is discrete.
\end{proposition} 

This immediately follows from Frol\'\i k's general theorem that the fixed point set 
of any surjective self-homeomorphism of an extremally disconnected 
space is clopen\footnote{Frol\'\i k proved this theorem for 
compact extremally disconnected spaces 
and not necessarily surjective self-homeomorphisms; in the surjective case, the theorem is extended 
to noncompact spaces by considering their Stone--\v Cech compactifications, which are 
always extremally disconnected for extremally disconnected spaces (this and other fundamental 
properties of extremally disconnected spaces can be found in the book~\cite{G-J}).}~\cite{Frolik}: 
it suffices to consider the self-homeomorphism of $X\times X$ defined by $(x, y)\mapsto (y,x)$.

\medskip

Thus, there exist (under CH) filters $\mathscr F$ on $\omega$ for which $(\omega_{\mathscr F})^2$ 
is not contained in $B(\omega_{\mathscr F})$ as a subspace. However, in the simplest case where 
$\mathscr F$ is the Fr\'echet filter (i.e., $\omega_{\mathscr F}$ is a convergent sequence), 
not merely does $B(\omega_{\mathscr F})$ contain $(\omega_{\mathscr F})^n$ but it is 
topologically isomorphic to $B(\omega_{\mathscr F})^n$ for all $n$ by virtue of 
Theorem~\ref{oplus} 
and the fact that a convergent sequence 
is $B$-equivalent to the disjoint union of two convergent 
sequences, which can be demonstrated as follows.

Any $M$-equivalent spaces are $A$-equivalent, and any $A$-equivalent spaces are $B$-equivalent,  
because $A(X)$ ($B(X)$) is the quotient of $F(X)$ ($A(X)$) by an algebraically determined subgroup 
not depending on $X$. Therefore, all known sufficient conditions for $M$- and $A$-equivalence 
(see, e.g.,~\cite{Graev1948, Graev1950, Okunev, Arh1980, Tkachuk} remain 
valid for $B$-equivalence. In particular, if $X_0$ is a space, $K$ is a retract of $X_0$, 
$X$ is the space obtained by adding an isolated point to $X_0$, and $Y= X_0/K\oplus K$, then 
$X$ and $Y$ are $M$-equivalent~\cite[Theorem~2.4]{Okunev}. This immediately 
implies the required $B$-equivalence of a convergent sequence $S$ and the disjoint union 
$S\oplus S$ of two convergent sequences: 
it suffices to take $S\oplus S$ for $X_0$ and $X$ and the two-point set of the two
limit points in $S\oplus S$ for~$K$. 

However, there exist $B$-equivalent spaces  
which are neither $F$- nor $A$-equivalent. Genze, Gul'ko, and Khmyleva obtained necessary 
and sufficient conditions for infinite initial segments of ordinals to be $F$-, $A$-, and 
$B$-equivalent~\cite{Genze-et-al} (see also~\cite{Genze-et-al2}). It turned out that the criteria 
for $F$- and $A$-equivalence are the same, and the criterion for $B$-equivalence differs from them; 
see~\cite{Genze-et-al} for details. 

Finally, the following theorem shows that there is also a fundamental difference 
between free groups of the varieties of Abelian and 
Boolean linear topological groups. 

\begin{theorem}
The free Boolean linear topological 
group of any strongly zero-dimensional pseudocompact space is precompact.
\end{theorem}

\begin{proof}
Let $X$ be a strongly zero-dimensional pseudocompact space. 
As mentioned in the preceding section, a base of neighborhoods of zero in $B^{\mathrm{lim}}(X)$ 
is formed by subgroups of the form
$$
\langle 
U(\gamma)\rangle = \Bigl\{\sum_{i=1}^n (x_i + y_i)\colon n\in \omega,\
(x_i, y_i) \in U_i\in \gamma \text{ for }i\le n\Bigr\},
$$
where $\gamma$ in a (finite) disjoint open cover of $X$. 
Clearly, 
$$
\langle 
U(\gamma)\rangle = \Bigl\{\sum_{i=1}^{2n} x_i\colon n\in \omega,\
|\{i\le 2n\colon x_i \in U\}| \text { is even for each }U\in \gamma\Bigr\}.
$$
Every such subgroup has finite index. Therefore, $B(X)$ is precompact.
\end{proof}

This theorem is not true for Abelian groups; moreover, free Abelian linear groups 
are never precompact. Indeed, the group 
$$
A^{\mathrm c}(X)=
\Bigl\{\sum_{i=1}^n x_i^{\varepsilon_i}\colon n\in \mathrm N,\ 
\sum_{i=1}^n \varepsilon_i = 1\Bigr\}
$$ 
considered above is always open, being the preimage of 
the isolated point~$0$ under the homomorphism $A(X)\to \mathbb Z_2 = \{0, 1\}$ 
which extends the constant map $X\to \{0, 1\}$ taking 
everything to $1$. As already mentioned, $A^{\mathrm c}(X)$ has infinite index in $A(X)$. 

\section{Extremally Disconnected Groups}

There is an old problem of Arkhangel'skii on the existence in ZFC of a nondiscrete 
Hausdorff extremally disconnected topological group; it was posed in 1967~\cite{Arh} 
and has been extensively studied since then. Various authors also posed the countable 
version of Arhangel'skii's problem  (see, e.g., \cite[Problem~6]{Ponomarev-Shapiro} 
and \cite[Question~6.1]{Comfort-vanMill}):  
Does there exist a ZFC example of a countable nondiscrete extremally disconnected topological 
group? The contable version of the problem has been solved in the negative 
in the quite recent joint paper~\cite{edgroups} by Reznichenko and the author; the uncountable 
case still persists. 

The first consistent example of a nondiscrete extremally disconnected group 
was constructed as early  as in 1969 by Sirota~\cite{Sirota};  more examples were 
constructed in~\cite{Louveau, Malykhin, Malykhin2, Zel1, Zel2}. We refer the reader interested 
in extremally disconnected groups and, in general, topological groups with extremal properties 
(maximal, nodec, and so on) to Zelenyuk's book~\cite{ProtQ}. 

A space $X$ is said to be \emph{extremally disconnected} 
if the closure of each open set in this space 
is open, or, equivalently, if any two disjoint open sets have disjoint closures. 
In particular, the space $X_{\mathscr F}$ associated with a filter $\mathscr F$ 
is extremally disconnected if and only 
if $\mathscr F$ is an ultrafilter. Clearly, all extremally disconnected spaces are 
zero-dimensional. 
The most fundamental properties of extremally disconnected spaces can be found 
in the book~\cite{G-J}. Much useful information (especially in the topological-algebraic context) 
is contained in~\cite{Arh-ed}. The central place 
in  the theory of extremally disconnected topological groups is occupied by Boolean topological 
groups because of the following beautiful theorem of Malykhin.

\begin{theorem}[Malykhin~\cite{Malykhin}]
Any extremally disconnected group contains an open (and therefore closed) Boolean subgroup.
\end{theorem}

This theorem follows from Frol\'\i k's  fixed-point theorem
mentioned at the end of the preceding section.  
In~\cite{Malykhin} Malykhin reproved Frol\'\i k's theorem for 
the particular self-homeomorphism $g \mapsto g^{-1}$; its fixed point set $U$ 
is an open neighborhood of the identity element, and the subgroup $H$ generated by 
an open neighborhood $V$ of the identity for which 
$V^2\subset U$ is as required. Indeed, any $u, v\in V$ commute: $uv\in U$ 
and hence $vu= vu(uvuv)=(v(uu)v)uv= uv$. Therefore, each $h\in H$ belongs to $U$, because 
$h = v_1v_2\dots v_n$ for some $v_1, v_2, \dots, v_n \in V$, and the $v_i$ commute (and belong to 
$U$), so that $h^2=v_1^2v_2^2\dots v_n^2$ equals the identity element. 

Thus, in the theory of extremally disconnected groups only Boolean groups matter. However, 
as we shall see later on, a nondiscrete free Boolean topological group cannot be extremally 
disconnected in ZFC~\cite{Sipa-free-ed}; moreover, in the case of countable groups, this 
it true for any (not necessarily free) group topologies~\cite{edgroups}. 

Extremal disconnectedness in groups is closely related to properties of 
countable discrete subsets of these groups. Thus, the presence of 
a countable nonclosed discrete set in an extremally disconnected group implies the existence 
of a $P$-ultrafilter~\cite{Zelenyuk06}, and the proof of the nonexistence in ZFC of 
countable nondiscrete extremally disconnected groups is based on the construction of two 
two disjoint discrete sequences with the same unique limit point under the assumption that there 
are no rapid ultrafilters~\cite{edgroups}. Below we present yet another observation concerning 
discrete subsets, more relevant to the context of free topological groups. 

\begin{theorem}
\label{li}
\begin{enumerate}
\item[(i)]
If $G$ is a hereditarily extremally disconnected Boolean \tolerance1000 group, 
then any closed linearly independent subset of 
$G$ contains at most one nonisolated point. 
\item[(ii)]
If $G$ is an extremally disconnected Boolean group, then any countable closed 
linearly independent subset of $G$ contains at most one nonisolated point. 
\end{enumerate}
\end{theorem}

A source of examples of hereditarily extremally disconnected spaces is provided by the 
following simple observation.

\begin{remark}[{see \cite[Exercise~6.2.G\,(c)]{Engelking}}]
\label{hed}
Any hereditarily normal extremally disconnected space is hereditarily extremally  
disconnected. Indeed, suppose that 
$X$ is a hereditarily normal extremally disconnected space. Let us show that 
any $Y \subset X$ is extremally disconnected. We can assume that $Y$ is closed in $X$, because, 
obviously, any dense subspace of an extremally disconnected space is extremally disconnected.  
We must show that the closures (in $Y$ or in $X$, there is no difference) of any 
disjoint sets $U$ and $V$ which are open in $Y$ are disjoint. Note that the closures $U'$ 
and $V'$ of $U$ and $V$ in the open subspace $Z=X\setminus (\overline U\cap \overline V)$ of $X$ 
are disjoint. Since $X$ is hereditarily normal, there exist disjoint open (in $Z$ and, therefore, 
in $X$) sets $U''\supset U'$ and $V''\supset V'$. Their closures in 
$X$ cannot intersect, because $X$ is extremally disconnected, and hence the closures in $Y$ 
of the smaller sets $U$ and $V$ do not intersect either.
\end{remark}

In the proof of Theorem~\ref{li} and later on we use the following obvious fact.  

\begin{remark}
\label{ched}
If countable sets $A$ and $B$ in a regular space $X$ are separated (i.e., each of them 
is disjoint from the other's closure), then they have disjoint open neighborhoods. Indeed, 
numbering the elements of $A$ and $B$ as $a_1,a_2,\dots$ and $b_1, b_2,\dots$, respectively, 
we can construct neighborhoods $U_i$ of $a_i$ and $V_i$ of $b_i$ so that 
each $\overline U_i$ is disjoint from $\overline B$ and from all $\overline{V_j}$ with $j\le i$ and 
each $\overline V_i$ is disjoint from $\overline A$ and from all $\overline{U_j}$ with $j\le i$. 
Clearly, $\bigcup U_i$ and $\bigcup V_i$ are disjoint open neighborhoods of $A$ and $B$.
\end{remark}

\begin{proof}[Proof of Theorem~\ref{li}]
(i) Let $A\subset G$ be a linearly independent subset of $G$. Suppose that $a\in A$ 
and $b\in A$ are 
distinct limit points of $A$. Take their disjoint closed neighborhoods $U\ni a$ and $V\ni b$. 
Since $A$ is linearly independent and closed, it follows that 
$a+ (V\cap A)\cap b + (V\cap A)= \{a + b\}$ and the 
sets $a+ (V\cap A)$ and $b + (V\cap A)$ are closed in $G$. 
Therefore, each of the disjoint sets $A'=a + ((V\setminus \{b\})\cap A)$ and 
$B'=b + ((U\setminus \{a\})\cap A)$ is open  
in the subspace $a+(V \cap A)\cup  b+(U \cap A)$  of $G$; 
obviously, $a+b$ belongs to the closure of 
each of them, which contradicts the hereditary extremal disconnectedness 
of $G$.

(ii) Arguing as in (i), we obtain countable sets $A'=a + ((V\setminus \{b\})\cap A)$ and 
$B'= b + ((U\setminus \{a\})\cap A)$, which are separated in $G$ and, therefore, 
have disjoint open neighborhoods $U$ and $V$ (see Remark~\ref{ched}). 
We have $\overline U\cap \overline V \supset 
\overline {A'}\cap \overline {B'} \ni a + b$, which contradicts the extremal disconnectedness 
of~$G$.
\end{proof}

\begin{corollary}
\label{cor-edfilter}
If $X$ is a nondiscrete countable space for which the free Boolean topological group $B(X)$ 
is extremally disconnected, then $X$ is the space $\omega_{\mathscr U}$ associated with 
an ultrafilter $\mathscr U$ on~$\omega$. 
\end{corollary}

Indeed, $X$ must be a filter space by Theorem~\ref{li}, and $X$ must be extremally disconnected, 
because $B(X)$ is countable and, hence, hereditarily extremally disconnected. Therefore, the 
filter associated with $X$ must be an ultrafilter.

In fact, $\mathscr F$ must be a Ramsey ultrafilter (see the next theorem and Theorem~\ref{Ramsey} 
in the last section). 

A detailed insight into extremally disconnected free Boolean topological groups 
on filter spaces can be gained from Section~\ref{secforcing}; here we present more general 
considerations.
 
The following theorem is a direct consequence of Theorem~1 in~\cite{Sipa-free-ed}. 

\begin{theorem}
\label{t1}
Let $X$ be a topological space satisfying any of the following two conditions: 
\begin{enumerate}
\item[(i)] the free Boolean topological group $B(X)$ is extremally disconnected;
\item[(ii)] $\operatorname{ind} X = 0$ and the free Boolean linear topological group 
$B^{\mathrm{lin}}(X)$ is extremally disconnected. 
\end{enumerate}
Then either $X$ is a $P$-space or there exists a Ramsey ultrafilter on $\omega$.  
\end{theorem}

\begin{proof}
Let $\tau$ be the free group topology on $B(X)$ in case (i) and the free linear group topology on 
$B(X)$ in case (ii). Suppose that $X$ is not a $P$-space, that is, some point 
$x_0\in X$ has a countable 
family of open neighborhoods $U_i$ ($i\in \mathbb N $) such that the 
interior of their intersection does not contain $x_0$. 
We can assume that $U_1=X$, 
$U_{i+1}\subset U_i$ for $i\in \mathbb N $, and, moreover, all 
neighborhoods 
$U_i$ are clopen, because $X$ 
is zero-dimensional (as a subspace of the extremally disconnected space 
$(B(X), \tau)$). 

Consider the countable space $Y$ obtained as 
the quotient image of $X$ under the map contracting the closed 
set $C_0=\bigcap_{i\in \mathbb N } U_i$  and the 
clopen sets $C_i=U_{i}\setminus U_{i+1}$, $i\in \mathbb N$, to points; namely, 
$Y$ is the countable set 
$Y=\{y_i: i\in \omega \}$ with the quotient topology induced 
by the map $\varphi \colon X\to Y$ defined as 
$$
\varphi (x)=\begin{cases}
y_0 &\text{if $x\in \bigcap_{i\in \mathbb N } U_i=C_0$,}\\
y_{i} &\text{if $x\in U_i\setminus U_{i+1}=C_i$, $i\in \mathbb N$.}
\end{cases}
$$
Clearly, the point  $y_0$ is not isolated in the space $Y$ (while all 
of the other points $y_i$ are isolated). Thus, $Y$ is a space of the form $\omega_{\mathscr F}$ 
for some filter $\mathscr F$ on $\omega$. 

Let $\tau'$ be the free group topology on $B(Y)$ in case (i) and the free linear group topology on 
$B(Y)$ in case (ii). The quotient map $\varphi: X\to Y$ extends to a continuous homomorphism 
$\bar\varphi: (B(X), \tau) \to (B(Y), \tau')$. Since $\tau'$ is the strongest group 
(or linear group) topology inducing the topology of $Y$ on $Y$, it follows that this homomorphism 
is quotient and, hence, open (as any quotient homomorphism). Open maps preserve 
extremal disconnectedness (see~{\cite[Problem~177]{Arh-Pon}}); therefore, 
the group $(B(Y), \tau')$ is extremally disconnected. It remains to apply Theorem~\ref{Ramsey} 
of the last section.
\end{proof} 

This theorem was proved in~\cite{Sipa-free-ed} in a more general situation (for $B(X)$ with any 
topology such that all continuous maps $X\to \mathbb Z_2$ extend to continuous homomorphisms), 
and the proof given there is much more complicated, because the continuity 
of $\bar\varphi$ is not automatic and Theorem~\ref{Ramsey} does not apply in this general 
situation. 

Theorem~\ref{t1} has the following immediate consequence. 

\begin{corollary}
The nonexistence of nondiscrete extremally disconnected free Boolean \tolerance1000 topological and linear 
topological groups is consistent with~\textup{ZFC}.
\end{corollary}

Indeed, as is known, 
any extremally disconnected $P$-space of 
nonmeasurable\footnote{Recall that a cardinal $\kappa$ is \emph{measurable} 
if there exists an ultrafilter with the countable intersection property 
on a set of cardinality $\kappa$. Measurable cardinals do not exist in, e.g., 
$\mathrm{ZFC}+ (\mathrm{V}=\mathrm{L})$, while the 
consistency of their existence with ZFC has not been proved.}  
cardinality is discrete (see~\cite{Isbell}). On the other hand,   
the nonexistence of measurable cardinals and Ramsey ultrafilters 
is consistent with \textup{ZFC} (see~\cite{Enayat}).

\section{Free Boolean Groups on Filters on $\omega$}

We have already seen in the preceding sections that free Boolean groups on 
almost discrete countable spaces (associated with filters on $\omega$) exhibit quite interesting 
behavior. Moreover, they are encountered more often than it may seem at first glance. 

Consider any Boolean group $B(X)$ with countable basis $X$. 
As mentioned in Section~\ref{Preliminaries}, 
this group is (algebraically) isomorphic to  the direct sum (or, in topological terminology, 
$\sigma$-product)
$\bigoplus^{\aleph_0}\mathbb Z_2$ of countably many copies of $\mathbb Z_2$. 
There is a familiar natural topology on this $\sigma$-product, 
namely, the usual product topology; let us denote it by 
$\tau_{\mathrm{prod}}$. This topology induces 
the topology of a convergent sequence on $X\cup\{0\}$ 
(where $0$ denotes the zero element of $B(X)$) and is metrizable; 
therefore, it never coincides with the topology 
$\tau_{\mathrm{free}}$ of the free Boolean topological group on $X$.  Moreover, 
$\tau_{\mathrm{prod}}$ is contained in $\tau_{\mathrm{free}}$ only when $X=\omega_{\mathscr F}$ 
for some filter (recall that we assume all filters to be free, i.e., contain the filter of 
cofinite sets, and identify the nonisolated points of the associated spaces with the 
zeros of their free Boolean groups). 
On the other hand, any countable space is zero-dimensional; therefore, any 
countable free Boolean topological group contains a sequence of subgroups with trivial intersection 
(see Theorem~\ref{free-linear}). The following statement shows that the topology of 
any group with this property contains the product topology $\tau_{\mathrm{prod}}$ associated with 
some basis.

\begin{theorem}[\hspace{1sp}{\cite[Lemma~2]{Sipa-2015}}]    
\label{subgroups}
Let $G$ be a countable nondiscrete  Boolean topological group which contains 
a family of open subgroups $G_i$ with trivial intersection.
Then there exists a basis of $G$ 
such that the isomorphism $G\to \bigoplus^{\aleph_0}\mathbb Z_2$ 
taking this basis to the canonical 
basis of $\bigoplus^{\aleph_0}\mathbb Z_2$ is continuous with respect to the product 
topology on $\bigoplus^{\aleph_0}\mathbb Z_2=\sigma (\mathbb Z_2)^{\aleph_0}$. 
\end{theorem}

\begin{proof}
We treat $G$ as a vector space  over the field $\mathbb Z_2$ and 
the $G_i$ as its subspaces. To prove the lemma, it suffices to construct 
a basis $E=\{e_n:  n\in\omega\}$ 
such that, for 
every $i\in \omega$,  there exists a $J_i\subset \omega$ for which $G_i = 
\langle e_n:n \in J_i\rangle $. Indeed, if $E$ is such a basis, then the assumption 
$\bigcap G_i=\{\mathbf{0}\}$ implies $\bigcap J_i = \varnothing$, and all linear spans $\langle 
e_k: k \ge n\rangle $ (which form a base of neighborhoods of zero in the product topology 
associated with the basis $E$) are open as subgroups with nonempty interior. 

In each (nontrivial) quotient space $G_i/G_{i+1}$, we 
take a basis $\{\varepsilon_\alpha: \gamma\in I_i\}$, where 
$|I_i|= \dim G_i/G_{i+1}$, and let 
$e_\alpha$ be representatives of 
$\varepsilon_\gamma$ in $G_i$. We assume the 
(at most countable) index sets $I_i$ to be well ordered and disjoint,  
let $I=\bigcup_{i\in \omega}I_i$, and endow $I$ with the lexicographic 
order (for $\alpha, \beta\in I$, we say that 
$\alpha<\beta$ if $\alpha\in I_i$, 
$\beta\in I_j$, and either $i<j$ or $i=j$ and $\alpha< \beta$ 
in $I_i$). For any $i\in \omega$ and $\alpha\in I_i$, we define 
$H_\alpha$ to be the subspace of $G$ spanned by $\{e_\beta: \alpha 
\in I_i, \beta\ge \alpha\}$ and $G_{i+1}$. Thus, 
$H_\beta$ is defined for each $\beta\in I$; moreover, if $\beta, \gamma \in I$ 
and $\beta < \gamma$ in $I$, then $H_\beta\supset H_\gamma$, and 
if $\alpha$ is the least element of $I_i$, then $H_\alpha=G_i$. 
This means that the subspaces $H_\alpha$ form a decreasing (with respect 
to the order induced by $I$) chain of subspaces refining the chain 
$G_0\supset G_1\supset\dots$\,.  Note that the lexicographic order on $I$ 
is a well-order, so for each $\alpha\in I$, its immediate successor 
$\alpha+1$ is  defined; by construction, we have 
$\dim H_{\alpha}/H_{\alpha+1}= 1$ for every $\alpha \in I$. 

Clearly, it suffices to construct a basis 
$E'=\{e'_\alpha: \alpha\in I\}$ with the property
$$
e'_\alpha\in H_\alpha\setminus H_{\alpha+1}
 \qquad\mbox{for every}\quad  \alpha\in I; 
\eqno{(\star)}
$$
the required basis $E$ is then obtained by reordering $E'$. Moreover, 
property~$(\star)$ ensures the linear independence of $E'$, so it is sufficient 
to construct a set of vectors spanning $G$ with this property. 

Take any basis $E''= \{e''_n: n\in \omega\}$ in $G$. We construct $E'$ 
by induction on $n$. 

Let $\alpha_0$ be the (unique) element of $I$ for which $e''_0\in 
H_{\alpha_0}\setminus H_{\alpha_0+1}$. We set $e'_{\alpha_0}=e''_0$. 

Suppose that $k$ is a positive integer 
and we have already defined elements $\alpha_i\in I$ 
and  vectors $e'_{\alpha_i}$ for $i<k$ so that 
$e'_{\alpha_i}\in H_{\alpha_i}\setminus H_{\alpha_i+1}$ and 
$\langle e'_{\alpha_0}, \dots, e'_{\alpha_{k-1}}\rangle = 
\langle e''_0,\dots, e''_{k-1}\rangle $. Take the (unique) $\alpha\in I$ 
for which $e''_k\in H_{\alpha}\setminus H_{\alpha+1}$. If $\alpha$ 
is unoccupied, that is, $\alpha \ne \alpha_i$ 
for any $i<k$, then we set $\alpha_k=\alpha$ and $e'_{\alpha_k}=e''_k$. If 
$\alpha$ is already occupied, i.e., $\alpha=\alpha_m$ for some $m<k$, then 
we take the (unique) $\beta\in I$ for which $e''_k+e'_{\alpha_m}\in   
H_{\beta}\setminus H_{\beta+1}$. (Clearly, $\beta>\alpha$, because 
$\dim H_{\alpha}/H_{\alpha+1}=1$ and $e''_k, e'_{\alpha_m} 
\in H_{\alpha}\setminus H_{\alpha+1}$). If $\beta$ is unoccupied, 
then we set $\alpha_k= \beta$ and $e'_{\alpha_k}=e''_k+e'_{\alpha_m}$; if 
$\beta= \alpha_l$ for $l<k$, then we take the (unique) $\gamma \in I$ for 
which $e''_k+e'_{\alpha_m}+e'_{\alpha_l}\in   
H_{\gamma}\setminus H_{\gamma+1}$ (clearly, $\gamma>\beta>\alpha$), and 
so on.  Only finitely many ($k-1$) indices from $I$ are occupied; therefore, 
after finitely many steps, we obtain $e''_k+e'_{\alpha_m}+e'_{\alpha_l}+
\dots + e'_{\alpha_s}\in H_{\delta}\setminus H_{\delta+1}$ for an unoccupied 
index $\delta$. We set $\alpha_k=\delta$ and $e'_{\alpha_k}= 
e''_k+e'_{\alpha_m}+e'_{\alpha_l}+\dots +e'_{\alpha_s}$. 

As a result, we obtain a set of vectors $E'=\{e'_{\alpha_n}:n\in \omega\}$ 
such that  
$e'_{\alpha_n}\in H_{\alpha_n}\setminus H_{\alpha_n+1}$ and 
$\langle e'_{\alpha_0}, \dots, e'_{\alpha_n}\rangle =
\langle e''_0, \dots, e''_n\rangle $ for every $n\in \omega$. The 
latter means that $E'$ spans $G$, because so does the basis 
$E''$. 

Formally, it may happen that not all of the indices $\alpha\in I$ 
are occupied, that is, $\{\alpha_n: n\in\omega\}=J\subsetneq I$. In 
this case, we take arbitrary  vectors $e'_\alpha\in H_\alpha\setminus H_{\alpha+1}$ for $\alpha\in 
I\setminus J$ and put them to $E'$. The set $E'$ thus enlarged 
satisfies condition $(\star)$ and is therefore linearly independent; 
thus, it cannot differ from the initial $E'$, because the latter 
spans $G$, and $J$ in fact coincides with $I$, i.e., 
$E'=\{e'_{\alpha_n}:n\in \omega\}=\{e'_\alpha:\alpha\in I\}$. 
This completes the proof of the lemma.  
\end{proof}

Theorem~\ref{subgroups} implies that 
any countable Boolean topological group containing a family of open subgroups with trivial 
intersection (in particular, any free Boolean topological or linear topological group on a 
countable space) has  a discrete or almost discrete closed basis. It turns out that this 
assertion holds for all countable Boolean topological groups. Namely, the following theorem is 
valid.

\begin{theorem}
Any countable Boolean topological group $G$ has either a discrete closed basis 
or a discrete basis for which $0$ is a unique limit point. 
In the latter case, $G$ is a continuous isomorphic image of the free Boolean topological group 
$B(\omega_{\mathscr F})$ on the space $\omega_{\mathscr F}$ associated with a filter $\mathscr F$ 
on $\omega$. 
\end{theorem}

\begin{proof}
On any countable topological group a continuous norm can be defined 
(see, e.g., \cite{Graev1950} or~\cite{Arhangelskii79}). Take any basis $\{e_1, e_2, \dots\}$ in 
$G$ and let $\|\cdot\|$ be a continuous norm on $G$. Consider a new basis $\{e'_1, e'_2, 
\dots\}$ defined by induction as follows:
\begin{itemize} 
\item 
$e'_1= e_1$; 
\item 
if $n\in \mathbb N$ and $e'_1, e'_2, \dots, e'_{n}$ are already defined, then $e'_{n+1}$ is a 
word in the alphabet $\{e'_1, e'_2, \dots, e'_{n}, e_{n+1}\}$ with minimum norm 
(if there are several such words, then we take any of them). 
\end{itemize}

For $k\in \omega$,  we denote the set of all words in 
$\{e'_1, e'_2, \dots\}$ of reduced length precisely $k$ by $G_{=k}$ and 
use the old notation $G_{k}$ for the set of all words of reduced length at most $k$; thus, 
$G_{k}=\bigcup_{m\le k}G_{=m}$. 

\begin{lemma}
\label{propbases1}
Let $w=e'_{i_1} + e'_{i_2} +\dots + e'_{i_n}$, where $n, i_1, \dots, i_n\in \mathbb N$ and 
$i_1< i_2< \dots < i_n$. Then 
$$
\|e'_{i_{n-k}}\|\le 2^k \|w\|\qquad \text{for each}\quad k=0, \dots, n-1.
$$
\end{lemma}

\begin{proof}
We argue by induction on $k$. The required inequality for $k=0$ follows from the definition of 
$e'_{i_n}$. Suppose that $l>0$ and the inequality holds for all $k<l$. We have $\|e'_{i_n}\|\le 
\|w\|$, \dots, $\|e'_{i_{n-l+1}}\|\le 2^{l-1}\|w\|$; by the triangle inequality for norms, we have 
$$
\|e'_{i_{n-l+1}}+e'_{i_{n-l+2}}+\dots +e'_{i_{n}}\|\le (2^{l-1}+2^{l-2}+\dots+1)\|w\|,
$$
and again applying the triangle inequality, we obtain 
$$
\|e'_{i_{1}}+e'_{i_{2}}+\dots +e'_{i_{l}}\|\le \|w\|+
\|e'_{i_{n-l+1}}+e'_{i_{n-l+2}}+\dots +e'_{i_{n}}\|\le 2^k\|w\|.
$$
\end{proof}

\begin{lemma}
\label{propbases2}
Each set $G_{=n}$ is discrete in $G$ with respect to the norm $\|\cdot\|$.
\end{lemma}

\begin{proof}
For $n=0$, the assertion is trivial. Suppose that $n>0$ and take any $w\in G_n$; we have $w= 
e'_{i_1}+\dots + e'_{i_n}$, where $i_1<\dots< i_n$. Let  $d= 
\min_{k<n}\left\{\frac{\|e'_{i_k}\|}{2^{2n}}\right\}$. Suppose that $w'\in G_n$, i.e., $w' = 
e'_{j_1}+\dots + e'_{j_n}$ for some $j_1<\dots< j_n$, and $w'$ belongs to the $d$-neighborhood of 
$w$ with respect to the norm $d$, i.e., $\|w+w'\|<d$. If the word $w''=w+w'$ 
contains a letter $x$, then $x$ equals $e'_{i_r}$ or $e'_{j_r}$ for some $r\le n$, and the 
length of $w''$ does not exceed $2n$; thus, by Lemma~\ref{propbases1}, we have $\|x\|\le 
2^{2n}\|w''\|$, whence $\|w''\|\ge \frac{\|x\|}{2^{2n}}$. By the definition of $d$ 
$x$ cannot equal $e'_{i_r}$, i.e., all letters $e'_{i_r}$ must be cancelled in the word 
$w''=w+w'$, which means that $w = w'$. Since the norm $\|\cdot\|$ on $G$ is continuous, it 
follows that the $d$-neighborhood of $w$ contains no elements of $G_{=n}$ except $w$. 
\end{proof}

\begin{lemma}
\label{propbases3}
Each set $G_{n}$ is closed in $G$ with respect to the norm $\|\cdot\|$. 
\end{lemma}

\begin{proof}
As in the proof of Lemma~\ref{propbases2}, suppose that $n>0$, take any $w= 
 e'_{i_1}+\dots + e'_{i_n}\in G_n$, where $i_1<\dots< i_n$, and let  $d= 
\min_{k<n}\left\{\frac{\|e'_{i_k}\|}{2^{2n}}\right\}$. We show that,  for $k< n$, the 
$d$-neighborhood of $w$ with respect to $d$ does not intersect $G_{k}$. Take any $w' 
= e'_{j_1}+\dots + e'_{j_k}$, where $j_1<\dots< j_k$. The word $w''=w+w'$ contains at least one 
letter $e'_{i_r}$ with $r\le n$, because $k<n$, and the length of $w''$ does not exceed $2n$; by 
Lemma~\ref{propbases1}, we have $\|e'_{i_r}\|\le 2^{2n}\|w''\|$, whence $\|w''\|\ge 
\frac{\|e'_{i_r}\|}{2^{2n}}\ge d$. This means that the 
$d$-neighborhood of $w$ contains no elements of $G_{k}$. 
\end{proof}

We proceed to prove the theorem. 
Note that the words of length~1 with respect to any given basis are precisely the basis elements. 
Therefore, by Lemma~\ref{propbases2}, the set $G_{=1} = \{e'_1, e'_2, \dots\}$ is discrete in 
$G$, and by Lemma~\ref{propbases3} the set $G_{1}=\{e'_1, e'_2, \dots\}\cup \{0\}$ is 
closed. This means that either $\{e'_1, e'_2, \dots\}$ is closed (and discrete) in $G$ or the 
subspace $\{e'_1, e'_2, \dots\}\cup \{0\}$ of $G$ is homeomorphic to a space of 
the form $\omega_{\mathscr F}$, where  $\mathscr F$ is a filter on $\omega$, and the homeomorphism 
$\{e'_1, e'_2, \dots\}\cup \{0\}\cong \omega_{\mathscr F}$ induces an algebraic isomorphism between 
$G$ and $B(\omega_{\mathscr F})$. Since the free group topology of $B(\omega_{\mathscr F})$ is the 
strongest group topology inducing the given topology on $\omega_{\mathscr F}$, it follows that the 
topology of $G$ (which induces the same tology on the topological copy $\{e'_1, e'_2, \dots\}\cup 
\{0\}$ of $\omega_{\mathscr F}$ in $G$) is coarser than that of $B(\omega_{\mathscr F})$. 
\end{proof}

\begin{remark}
Lemma~\ref{propbases1} (with obvious modifications) and Lemmas~\ref{propbases2} 
and~\ref{propbases2} are valid for any normed countable-dimensional vector space 
over a finite field.
\end{remark}

Spaces of the form $\omega_{\mathscr F}$ are one of the rare examples where the free Boolean 
topological group is naturally embedded in the free and free Abelian topological groups 
as a closed subspace. The embedding of $B(\omega_{\mathscr F})$ into 
$A(\omega_{\mathscr F})$ is defined simply by $x_1+x_2+ \dots +x_n\mapsto x_1+x_2+ \dots +x_n$ 
(for the Graev free groups, which coincide with Markov ones for such spaces), and 
the embedding into $F(\omega_{\mathscr F})$ is $x_1+x_2+ \dots +x_n\mapsto x_1x_2 \dots x_n$, provided 
that $x_1<x_2<\dots<x_n$ (in $\omega$). These embeddings take $B(\omega_{\mathscr F})$ to 
\begin{multline*}
A=\{x_1+x_2+ \dots +x_n=(x_1-*)+(x_2-*)+ \dots +(x_n-*)\colon \\ 
n\in \mathbb N,\ x_i \in \omega\} \subset A(\omega_{\mathscr F})
\end{multline*}
and 
\begin{multline*}
F=\{x_1x_2 \dots x_n=x_1*^{-1}x_2*^{-1} \dots x_n*^{-1}\colon\\ 
n\in \mathbb N,\ x_i \in \omega,\ x_1<x_2< \dots<x_n\}\subset F(\omega_{\mathscr F}).
\end{multline*}
The topologies induced on $A$ and $F$ by $A(\omega_{\mathscr F})$ and $F(\omega_{\mathscr F})$ 
are easy to describe; the restrictions of base neighborhoods of the zero (identity) element 
to these sets are determined by sequences of open covers of 
$\omega_{\mathscr F}$ (i.e., of neighborhoods of the nonisolated point $*$) in the same manner 
as in description~II (see~\cite{VINITI}). The rigorous proof of the homeomorphism of $A$, $F$, 
and $B(\omega_{\mathscr F})$ is obvious but tedious, and we omit it.  

\section{Free Boolean Topological Groups and Forcing}
\label{secforcing}

As mentioned at the end of Section~\ref{secdescriptions}, for any filter 
$\mathscr F$, the free Boolean group on $\omega_{\mathscr F}$ is simply $[\omega]^{<\omega}$ with 
zero $\emptyset$. Generally, a topology on any set is a partially ordered 
(by inclusion) family of subsets. Partial orderings of subsets of $[\omega]^{<\omega}$ have been 
extensively studied in forcing, and countable Boolean topological groups turn out to be closely 
related to them. In this section we shall try to give an intuitive explanation of this 
relationship. The basic definitions and facts related to forcing can be found in Kunen's 
book~\cite{Kunen} and Jech's book~\cite{Jech}.

Forcing is a method for extending models of set theory so as to include an object 
with desired properties. This is done by means of a partially ordered set $(\mathbb P, \le)$,   
referred to as a \emph{notion of forcing}, and a {generic set} $G$ consisting of compatible 
elements of $\mathbb P$ and meeting all dense subsets of $\mathbb P$  
(such sets never belong to the ground model, except in trivial cases of no interest). 
The method of forcing yields the minimal extension of $M$ containing $G$, called the 
\emph{generic extension}. 
The object with desired properties, which is the goal of the construction, 
is often simply~$\bigcup G$ or~$\bigcap G$. Figuratively, the desired propeties are, 
so to speak, a frame for an infinite picture, the elements of 
$\mathbb P$ are finite jigsaw pieces (which can never be fit together to make a picture 
large enough in the space where $\mathbb P$ lives), and $G$ is a set of compatible 
pieces that form a picture filling the frame but in a higher-dimensional space. 
Note that the design of $\mathbb P$ is as important as 
that of $G$, because $\mathbb P$ is responsible for preventing indesirable destructive effects 
of forcing, such as collapse of cardinals. 

Given two conditions $p, q \in \mathbb P$, $p$ is said to be \emph{stronger} 
than $q$ if $p \le q$. A partially ordered set $(\mathbb P, \le)$ is separative if, whenever 
$p \not\le  q$, there exists an $r \le p$ which is incompatible with $q$. Thus, any topology 
is a generally nonseparative notion of forcing, and the family of all regular open sets 
in a topology is a separative notion of forcing. Any separative forcing notion $(P, \le)$ 
is isomorphic to a dense subset of a complete Boolean algebra. Indeed,  
consider the set $\mathbb P\downarrow  p = \{q : q \le p\}$ for each $p \in \mathbb P$. 
The family $\{X \subset \mathbb P: (\mathbb P \downarrow p) 
\subset  X \text{ for every } p \in X\}$ generates a topology on $\mathbb P$. 
The complete Boolean algebra mentioned above is the algebra $\operatorname{RO}(\mathbb P)$
of regular open sets in this topology.

Two notions of forcing $\mathbb P$ and $\mathbb Q$ are said to be 
\emph{forcing equivalent} if the algebras 
$\operatorname{RO}(\mathbb P)$ and $\operatorname{RO}(\mathbb Q)$ are
isomorphic, or, equivalently, if $\mathbb P$ can be densely embedded in 
$\mathbb Q$ and vice versa (which means that $\mathbb P$  and $\mathbb Q$ produce the same 
generic extensions).

In the context of free Boolean groups on filters most interesting are two 
well-known notions of forcing, 
Mathias forcing and Laver forcing relativized to filters on~$\omega$. 

In \emph{Mathias forcing} relative to a filter $\mathscr F$ 
the forcing poset, denoted $\mathbb M(\mathscr F)$, 
is a pair $(s, A)$ consisting 
of a finite set $s\subset \omega$ and an (infinite) set $A\in \mathscr F$  
such that every element of $s$ is less than every element of $A$ in the ordering of $\omega$ (i.e., 
$\max s<\min A$). 
A condition 
$(t, B)$ is stronger than $(s,A)$ ($(t,B) \le (s,A)$) if $s\sqsubset t$, $B\subset A$, 
and $t \setminus s \subset  A$.

The poset in Laver forcing consists of subsets of the set $\omega^{< \omega}$ of ordered finite 
sequences in $\omega$. However, it is more convenient for our purposes to consider its 
modification consisting of subsets of $[\omega]^{< \omega}$. Thus, we restrict the Laver forcing 
poset to the set $\omega^{\uparrow < \omega}$ of strictly increasing finite sequences 
(this restricted poset is forcing equivalent to the original one) and note 
that the latter is naturally identified with $[\omega]^{< \omega}$. Below we give the 
definition of the corresponding modification of Laver forcing.

The definition of Laver forcing uses the notion of a Laver tree. 
A \emph{Laver tree} is a set $p$ of finite subsets of 
$\omega$
such that 
\begin{enumerate}
\item[(i)] $p$ is a tree (i.e., if $t\in p$, then $p$ contains any initial segment of 
$t$),
\item[(ii)] $p$ has a stem, i.e., a maximal node $s(p) \in p$ 
such that $s(p) \sqsubset t$ or $t \sqsubset  s(p)$ for all $t \in p$, and 
\item[(iii)] if $t \in p$ and 
$s(p) \sqsubset t$, 
then the set $\operatorname{succ}(t) = \{n\in \omega\colon n > \max t, \ t\cup \{n\} \in p\}$ 
is infinite.
\end{enumerate}
In \emph{Laver forcing} relative to $\mathscr F$ the poset, denoted 
$\mathbb L(\mathscr F)$, is the set of Laver trees $p$ such that 
$\operatorname{succ}(t)\in \mathscr F$ for any  
$t \in p$  with $s(p) \sqsubset t$,  
ordered by inclusion.

The Mathias and Laver forcings $\mathbb M (\mathscr F)$ and $\mathbb L(\mathscr F)$ 
determine two natural topologies on $[\omega]^{<\omega}$: 
the \emph{Mathias topology} $\tau_M$ generated by the base 
\begin{multline*}
\{[s,A] \colon 
s \in [\omega]^{<\omega},\ A\in \mathscr F,\ \max s<\min A\},\\
\text{where}\quad 
[s,A] = \{t \in [\omega]^{<\omega}\colon 
s \sqsubset t, \ t\setminus s \subset A\},
\end{multline*} 
and the \emph{Laver topology} $\tau_L$ generated by all sets $U\subset [\omega]^{<\omega}$ 
such that 
$$
t\in U \implies \{n > \max t\colon t\cup \{n\} \in U\}\in \mathscr F.
$$

It is easy to see that the Mathias topology is nothing but the topology of the free Boolean 
linear topological 
group on $\omega_{\mathscr F}$ (recall that linear groups are those with topology 
generated by subgroups): a base of neighborhoods of zero $\emptyset$ is formed by the sets 
$[\emptyset, A]$ with $A\in \mathscr F$, that is, by all subgroups generated by 
elements of $\mathscr F$. 

The neighborhoods of zero in the Laver topology are not so easy to describe explicitly; 
their recursive definition immediately follows from that given above 
for general open sets (the only additional condition $\emptyset \in U$ must be added). 
Thus, $U$ is an open neighborhood of zero if, first, $\emptyset \in U$; by definition, $U$ 
must also contain all $n\in A(\emptyset)$ for some $A(\emptyset)\in \mathscr F$ (moreover, $U$ 
may contain no other 
elements of size~1); for each of these $n$, there must exist an $A(n)\in \mathscr F$ such that 
$\min A(n)> n$ and 
$U$ contains all $\{n, m\}$ with $m\in A(n)$ (moreover, $U$
may contain no other elements of size~2); 
for any such $\{n, m\}$ (note that $m> n$) 
there must exist an $A(\{n,m\})\in \mathscr F$ such that $\min A(\{n, m\})>m$ 
and $U$ contains all $\{n, m, l\}$ with $l\in A(\{n,m\})$, and so on. Thus, 
each neighborhood of zero is determined by a family $\{A(s)\colon s\in [\omega]^{<\omega}\}$ 
of elements of $\mathscr F$. Clearly, the topology $\tau_L$ is invariant with respect to 
translation by elements of $[\omega]^{<\omega}$; 
upon a little reflection it 
becomes clear that $\tau_L$ is the maximal invariant topology on $[\omega]^{<\omega}$ in which the 
filter $\mathscr F$ converges to zero.\footnote{An invariant topology is a topology with respect 
to which the group operation is separately continuous; Boolean groups with an invariant topology 
are precisely quasi-topological Boolean groups. The convergence of $\mathscr F$ to zero 
means that $\mathscr F$ contains all neighborhoods of zero in $\tau_L$ restricted to $\omega$, 
i.e., that $\tau_L$ induces the initially given topology on $\omega_{\mathscr F}$.}
Since the free group topology is invariant as well, 
it is coarser than $\tau_L$. 

The Mathias topology is, so to speak, the uniform version of the Laver topology: a neighborhood 
of zero in the Laver topology 
determined by  a family $\{A(s)\in \mathscr F\colon s\in [\omega]^{<\omega}\}$ 
is open in the Mathias topology if and only if there exists 
a single $A\in \mathscr F$ such that $A(s)=A\setminus\{0, 1, \dots, \max s\}$ for each $s$. 
(In~\cite{Brendle} the corresponding relationship between Mathias and Laver forcings 
was discussed from a purely set-theoretic point of view.) Hence $\tau_M\subset \tau_L$. 

The topology of the free Boolean topological group on $\omega_{\mathscr F}$ 
occupies an intermediate position between the Mathias and the Laver topology: it is not so 
uniform as the former but more uniform than the latter. Neighborhoods of zero are determined not 
by  a single element of the filter (like in the Mathias topology) but by a family of elements of 
$\mathscr F$ assigned to $s\in [\omega]^{<\omega}$ (like in the Laver topology), 
but these elements depend only on the lengths of $s$. 

The following theorem shows that 
the Laver topology is a group topology only for 
special filters. This theorem  was proved in 2007 by Egbert Th\"ummel, who 
kindly communicated it, together with a complete proof, to the author. The symbols 
$\tau_{\mathrm{free}}$ and $\tau_{\mathrm{indlim}}$ in its statement denote the topology
of the free topological group $B(\omega_{\mathscr F})$ and the inductive limit topology of 
$B(\omega_{\mathscr F})$, respectively. We conventionally use the term \emph{selective filter} 
for a filter satisfying any of the equivalent conditions (iii)--(v) in 
Theorem~\ref{ramseychar}. Recall that, according to condition~(iv), if $\mathscr F$ is 
a selective filter, then any family $\{A_i: i \in\omega\}$, where $A_i\in \mathscr F$, has a 
\emph{diagonal intersection} in $\mathscr F$, that is, there exists 
a set $D\in \mathscr F$ such that $j \in  A_i$ whenever 
$i, j \in D$ and $i < j$.

\begin{theorem}[Th\"ummel, 2007]
\label{Thuemmel}
For any filter on $\omega$, the following conditions are equivalent:
\begin{enumerate}
\item[(i)] $\mathscr F$ is selective; 
\item[(ii)] $\tau_M = \tau_{\mathrm{free}} = \tau_{\mathrm{indlim}} = \tau_L$\textup;
\item[(iii)] $\tau_L$ is a group topology;
\item[(iv)] for any sequence of $A_i\in \mathscr F$ with $\min A_i>i$, $i\in \omega$, 
the set $U=\{\emptyset\}\cup\bigcup_{i\in \omega} [i, A_i]$ is open in $\tau_{\mathrm{free}}$.
 \end{enumerate}
\end{theorem}

Th\"ummel mentioned that, in his proof of the implication (i)$\implies$(ii) in this theorem, 
he used an argument known from forcing theory. The proof given below only slightly differs from 
Th\"ummel's. 

\begin{proof}
First, note that $\tau_M \subset \tau_{\mathrm{free}} \subset \tau_{\mathrm{indlim}} 
\subset \tau_L$. Indeed, the 
first two inclusions are obvious (recall that $\tau_M$ coincides with the free linear group 
topology), and the third one follows from Theorem~\ref{prop-linear} (or from the inclusion 
$\tau_{\mathrm{free}}\subset \tau_L$ noted above) and the observation that $\tau_L$ is the 
inductive limit of its restrictions to $B_n(\omega_{\mathscr F})$. 

Thus, to prove the implication (i)~$\Rightarrow$~(ii), it suffices to show that $\tau_M = \tau_L$ 
for any selective filter. Let $U$ be a neighborhood of $\emptyset$ in $\tau_L$. For each 
$i\in \omega$, we set 
$$
A_i=\bigcap\bigl\{\{n>\max s\colon s\cup \{n\}\in U\}\colon  s\in U,\ \max s \le i\bigr\}.
$$
Since the number of $s\in [\omega]^{<\omega}$ with $\max s\le i$ is finite, it follows that 
$A_i\in \mathscr F$. Take a diagonal intersection $D\in \mathscr F$ of the family 
$\{A_i\colon i\in \omega\}$. 
We can assume that $D\subset A_0$. Clearly, 
$[\emptyset, D]\subset U$, whence $U\in \tau_M$.

The implication (ii)~$\Rightarrow$~(iii) is trivial. 

Let us prove (iii)~$\Rightarrow$~(iv). Note that it follows from (iii) that 
$\tau_{\mathrm{free}}=\tau_L$, because $\tau_{\mathrm{free}}\subset \tau_L$ and 
$\tau_{\mathrm{free}}$ is the strongest group topology inducing the initially given topology on 
$\omega_{\mathscr F}$. It remains to note that any set of the form 
$\{\emptyset\}\cup\bigcup_{i\in \omega} [i, A_i]$, where $A_i\in \mathscr F$ and $\min A_i>i$, 
is open in $\tau_L$. 

We proceed to the last implication (iv)~$\Rightarrow$~(i). Take any 
sequence of $A_i\in \mathscr F$ with $\min A_i> i$, $i\in \omega$, and consider the set 
$U$ defined as in (iv). Since this is an open neighborhood of zero in the group topology 
$\tau_{\mathrm{free}}$, there exists an open neighborhood $V$ of zero (in $\tau_{\mathrm{free}}$) 
such that $V+V\subset U$. The set $D=\{i\in \omega\colon i\in V\}$ belongs to $\mathscr F$ 
(because $\tau_{\mathrm{free}}$ induces the initially given topology on $\omega_{\mathscr F})$ and 
is a diagonal intersection of $\{A_i\colon i \in\omega\}$. Indeed, if $i<j$ and $i, j\in D$, then
$i+j=\{i, j\}\in U$; thus, there exists a $k$ for which $\{i, j\}\in [k, A_k]$. 
The conditions $\min A_k >k$ and $i< j$ imply $k=i$. Therefore, $j\in A_i$. 
\end{proof}

Theorem~\ref{Thuemmel} is worth comparing with 
Ihoda and Shelah's theorem that  if $\mathscr F$ is a Ramsey ultrafilter, 
then $\mathbb M(\mathscr F)$ is forcing equivalent to
$\mathbb L(\mathscr F)$~\cite[Theorem~1.20~(i)]{JuSh89} (Mathias forcing is referred to as Silver 
forcing in~\cite{JuSh89}). 

Th\"ummel also noticed that Theorem~\ref{Thuemmel}, combined with Sirota's construction of a 
CH example of an extremally disconnected group, implies that, 
\emph{given a filter $\mathscr F$ on $\omega$, the free Boolean 
topological group $B(\omega_{\mathscr F})$ is extremally disconnected if and only if 
$\mathscr F$ is a Ramsey ultrafilter}. Below 
we prove this statement in a formally stronger form: we do not assume $\mathscr F$ to be an 
ultrafilter in the \emph{if} part. (Amazingly, the most immediate consequence of this 
stronger statement is that is is not actually stronger.)

Th\"ummel has never published these results, and 
the statement that, for an ultrafilter $\mathscr U$ on $\omega$, 
$B(\omega_{\mathscr U})$ is extremally disconnected if and only if 
$\mathscr U$ is Ramsey was rediscovered by Zelenyuk, 
who included it, among other impressive results, in his book~\cite{ProtQ} 
(see Theorem~5.1 in~\cite{ProtQ}). 

\begin{theorem}
\label{Ramsey}
\begin{enumerate}
\item[(i)] For any selective filter $\mathscr F$ on $\omega$, the free Boolean linear 
topological group $B^{\mathrm{lin}}(\omega_{\mathscr F})$ 
(and hence the free Boolean topological group $B(\omega_{\mathscr F})$) is extremally disconnected.
\item[(ii)] If $\mathscr F$ is a filter on $\omega$ for which 
$B^{\mathrm{lin}}(\omega_{\mathscr F})$ or $B(\omega_{\mathscr F})$ is extremally disconnected, 
then $\mathscr F$ is a Ramsey ultrafilter.
\end{enumerate}
\end{theorem}

\begin{proof} 
The proof of (i) is essentially contained in Sirota's construction 
of a (consistent) example of an extremally disconnected group~\cite{Sirota}. In~\cite{Sirota} 
Sirota introduced the notion of a $k$-ultrafilter on $\omega$ and proved that 
$B^{\mathrm{lin}}(\omega_{\mathscr U})$ is extremally disconnected for any $k$-ultrafilter 
$\mathscr U$. A $k$-ultrafilter is defined as an ultrafilter satisfying two conditions, 
one of which is precisely the selectivity condition (v) in Theorem~\ref{ramseychar}. 
We use only selectivity and do not assume our filter to be an ultrafilter; 
at that, thanks to Theorem~\ref{Thuemmel}, 
the proof presented below is far simpler than Sirota's original proof.  

\begin{lemma}
\label{Sirota1} 
If $U$ is an open set in $B^{\mathrm{lin}}(\omega_{\mathscr F})$ and $\emptyset \in 
\overline{\omega\cap U}$, then $\{\emptyset\}\cup U$ is open, i.e., $\emptyset \in \Int 
\overline{U}$. 
\end{lemma}

\begin{proof}
Note that $\emptyset\in \overline{\omega\cap U}$ if and only if $U\supset A$ for some $A\in 
\mathscr F$. Since $U$ is open in $\tau_M$, it must contain each $i\in A\cap \omega$ together with 
its open neighborhood. Thus, $U$ contains a set of the the form $[i, A_i]$, $A_i\in \mathscr F$, 
for each $i\in A$, and any set of the form  $\{\emptyset\}\cup\bigcup_{i\in 
A} [i, A_i]$, where $A_i\in \mathscr F$, is open in $\tau_L=\tau_M$. 
\end{proof}

\begin{lemma}
\label{Sirota2}
If $X\subset B^{\mathrm{lin}}(\omega_{\mathscr F})$ and $\emptyset \in 
\overline{X}\setminus X$, then $\emptyset \in \overline{\omega\cap \overline{X}}$. 
\end{lemma}

\begin{proof}
Suppose that, on the contrary, $\emptyset \notin \overline{\omega \cap \overline{X}}$. Then 
$\emptyset \in \overline{\omega \setminus \overline{X}}$ (because $\emptyset \in 
\overline{\omega}$). Let $U=B^{\mathrm{lin}}(\omega_{\mathscr F}) \setminus \overline{X}$. We have 
$\omega\cap U= \omega\setminus \overline {X}$, so that by Lemma~\ref{Sirota1}  $W= 
U\cup\{\emptyset\}$ is an open neighborhood of $\emptyset$. 
We also have $\emptyset \notin X$, $U\cap X = \emptyset$, and $W \cap X=\emptyset$, which 
contradicts the assumption $\emptyset \in \overline{X}$.
\end{proof}

\begin{lemma}
\label{Sirota3}
If $U$ is an open subset of $B^{\mathrm{lin}}(\omega_{\mathscr F})$, $k\in \mathbb N$, and 
$\emptyset \in \overline{B_k(\omega_{\mathscr F}) \cap U}$, then $\emptyset 
\in \Int \overline{U}$. 
\end{lemma}

\begin{proof}
We prove the lemma by induction on $k$. For $k=0$, $B_k(\omega_{\mathscr F})=\{\emptyset\}$, and we 
have $\emptyset \in U\subset \Int \overline{U}$. Suppose that $n> 0$ and the required assertion 
holds for all $k<n$. Let us prove that if $\emptyset \in \overline{B_n(\omega_{\mathscr F}) \cap 
U}$, then $\emptyset \in \Int \overline{U}$.  

First, we show that 
$\omega \cap \overline{B_n(\omega_{\mathscr F}) \cap U}\subset \Int \overline{U}$. Take any $i\in 
\omega$ such that $i\in \overline{B_n(\omega_{\mathscr F}) \cap U}$. An arbitrary open neighborhood 
of $i$ contains a neighborhood of the form $[i, A]$, where $A\in \mathscr F$, $\min A> i$, and 
we have $i\in \overline{[i, A]\cap B_n(\omega_{\mathscr F}) \cap U}$. Note that 
$(i+([i, A]\cap B_n(\omega_{\mathscr F})))\subset B_{n-1}(\omega_{\mathscr F})$, because 
$\min A > i$ and $i+i = \emptyset$. On the other hand, $\emptyset = i+i\in 
\overline{(i+([i, A]\cap B_n(\omega_{\mathscr F})))\cap (i+U)}$. By the induction hypothesis, we 
have $\emptyset \in \Int\overline{i+U}$, whence $i\in \Int\overline{U}$. 

Thus, $\omega \cap \overline{B_n(\omega_{\mathscr F}) \cap U}\subset \Int \overline{U}$. If 
$\emptyset \in U$, there is nothing to prove. Suppose that 
$\emptyset \notin U\supset B_n(\omega_{\mathscr F}) \cap U$. Then, by Lemma~\ref{Sirota2}, we have 
$\emptyset \in \overline{\omega \cap \overline{B_n(\omega_{\mathscr F}) \cap U}}\subset 
\overline{\omega \cap \Int \overline{U}}$, and Lemma~\ref{Sirota1} implies $\emptyset \in 
\Int\overline{\Int \overline{U}} = \Int \overline{U}$.
\end{proof}

To complete the proof of (i), it remains to recall that  
$\tau_M=\tau_{\mathrm{indlim}}$ for selective filters and, therefore, $\overline X = 
\bigcup_{k\in \omega} \overline{B_k(\omega_{\mathscr F}) \cap X}$ for any  $X\subset 
B^{\mathrm{lin}}(\omega_{\mathscr F})$. Thus, by Lemma~\ref{Sirota3}, we have 
$\emptyset \in \Int\overline{U}$ whenever $U$ is an open set and $\emptyset \in \overline U$. Since 
$B^{\mathrm{lin}}(\omega_{\mathscr F})$ is homogeneous, it follows that, for any open $U\subset  
B^{\mathrm{lin}}(\omega_{\mathscr F})$ and any $x\in \overline U$, we have $x\in \Int\overline{U}$, 
i.e., $\overline{U}$ is open. Thus, the free Boolean linear topological group 
$B^{\mathrm{lin}}(\omega_{\mathscr F})$ is extremally disconnected, and hence so is the free  
Boolean topological group $B(\omega_{\mathscr F})$, because these groups coincide for selective 
filters by Theorem~\ref{Thuemmel}\,(ii).

The proof of (ii) is based on the implication (iv)~$\Leftrightarrow$~(i) 
of Theorem~\ref{Thuemmel}: for any sequence of  $A_i\in \mathscr F$ with $\min A_i>i$, $i\in 
\omega$, the set $U=\bigcup_{i\in \omega} [i, A_i]$ is by definition open in both $\tau_M$ and 
$\tau_{\mathrm{free}}$, and $U'=\{\emptyset\}\cup\bigcup_{i\in \omega} [i, A_i]$ 
is closed in each of these topologies. Indeed, suppose that 
$s=\{i_1, i_2,\dots, i_n\}\notin U'$ and $i_1<i_2<\dots <i_n$. Then 
$i_k\notin A_{i_1}$ for some $k\in \{2, \dots , n\}$. Let $s'=s\setminus A_{i_1}$. Then 
$[s, A_{i_1}]$ is an open (in both topologies) neighborhood of $s$, and it does not 
intersect $U'$, because the least letter of each word in $[s, A_{i_1}]$ is $i_1$, and any such 
word contains $i_k\notin A_{i_1}$. Thus, $U'$ is the closure of $U$ in $\tau_M$ and 
$\tau_{\mathrm{free}}$. Therefore, $U'$ must be open in $\tau_{\mathrm{free}}$ 
(if $B(\omega_{\mathscr F})$ is extremally disconnected) or even in $\tau_M\supset 
\tau_{\mathrm{free}}$ (if $B^{\mathrm{lin}}(\omega_{\mathscr F})$ is extremally disconnected). 
In any case, the assertion (iv)~$\Leftrightarrow$~(i) of Theorem~\ref{Thuemmel} implies 
that $\mathscr F$ is a selective filter. It remains to apply Corollary~\ref{cor-edfilter}. 
\end{proof} 

Theorem~\ref{Ramsey} immediately implies the following (most likely, known) statement 
concerning filters. 

\begin{corollary}
\label{selective}
Each filter satisfying any of the equivalent conditions \textup{(iii)--(v)} in Theorem~\ref{ramseychar} 
is a Ramsey ultrafilter.
\end{corollary}

Free Boolean topological and free Boolean 
linear (that is, Mathias) topological groups on spaces associated with filters, as well 
as Boolean groups with other topologies determined by filters, are the main tool in the 
study of topological groups with extreme topological properties (see~\cite{ProtQ} and the 
references therein). However, free Boolean (linear) topological groups on filters 
arise also in more ``conservative'' domains. We conclude with mentioning an instance of 
this kind. 

The most elegant (in the author's opinion) example of a countable nonmetrizable Fr\'echet--Urysohn 
group was constructed by Nyikos in~\cite{Nyikos} under the relatively mild 
assumption $\mathfrak p = \mathfrak b$ (Hru\v s\'ak and Ramos-Garc\'\i a 
have recently proved that 
such an example cannot be constructed in ZFC~\cite{Hrusak}). 

It is clear from general considerations that test spaces most convenient for studying convergence 
properties which can be defined pointwise (such as the Fr\'echet--Urysohn property and 
the related $\alpha_i$-properties) are countable 
almost discrete spaces (that is, spaces 
of the form $\omega_{\mathscr F}$), 
and the most convenient test groups for studying such properties in topological groups  are 
those generated by such spaces, simplest among which are free Boolean linear 
topological groups. Thus, it is quite natural that Nyikos' example 
is $B^{\mathrm{lin}}(\omega_{\mathscr F})$ for a very cleverly constructed filter $\mathscr F$. 
In fact, he constructed it on $\omega\times \omega$ (which does not make any 
difference, of course) as the 
set of neighborhoods of the only nonisolated point in a $\Psi$-like space defined by using graphs 
of functions $\omega\to \omega$ from a special family. In the same paper Nyikos proved many 
interesting convergence properties of groups $B^{\mathrm{lin}}(\omega_{\mathscr F})$ for arbitrary 
filters $\mathscr F$ on $\omega$. We do not give any more details here: the 
interested reader will gain much more benefit and pleasure from reading Nyikos' original 
paper.

\end{document}